\documentclass{amsart}

\usepackage{paper_style}
\usepackage[csfont, paper]{std_math}
\usepackage{local}

\IfFileExists{./\foldername/abstract.tex}{}{}

\newif\ifourcomments
\ourcommentstrue

\ifourcomments
\newcommand{\kanote}[1]{\marginpar{\color{blue}\tiny [KA] #1}}
\newcommandx{\kaadd}[2][2=]{{\color{blue}{\tiny [KA]} #1} \marginpar{\color{magenta} [KA] #2}}
\newcommand{\manote}[1]{\marginpar{\color{magenta}\tiny [MA] #1}}
\newcommand{\maadd}[1]{{\color{magenta}{\tiny [MA]} #1} \marginpar{\color{magenta} [MA]}}
\newcommand{\awnote}[1]{\marginpar{\color{cyan}\tiny [AW] #1}}
\newcommand{\awadd}[1]{{\color{cyan}{ #1} \marginpar{\color{cyan} [AW]}}}

\else
\newcommand{\kanote}[1]{}
\newcommandx{\kaadd}[2][2=]{}
\newcommand{\manote}[1]{}
\newcommand{\maadd}[1]{}
\newcommand{\awnote}[1]{}
\newcommand{\awadd}[1]{}
\fi

\hypersetup{pdfauthor={Menny Aka, Konstantin Andritsch, Andreas Wieser}
	,pdftitle=Planes in quadratic 4-space
	,urlcolor=blue
	,citecolor=red
	,linkcolor=OliveGreen
	,colorlinks=true
}
\begin{document}
		
\title[Planes in quadratic $4$-space and associated shapes of lattices]{Planes in quadratic $4$-space and \\ associated shapes of lattices}

\author{Menny Aka}
\address[M.A.]{Department of Mathematics, ETH Zürich, Zürich, Switzerland}
\email{menny.akka@math.ethz.ch}

\author{Konstantin Andritsch}
\address[K.A.]{Department of Mathematics, ETH Zürich, Zürich, Switzerland}
\email{konstantin.andritsch@math.ethz.ch}

\author{Andreas Wieser}
\address[A.W.]{Institute for Advanced Study, 1 Einstein Drive, Princeton, NJ 08540, USA}
\email{awieser@ias.edu}

\thanks{M.A.~and K.A.~acknowledge the support of the SNF Grant 10003145, Equidistribution in Number Theory. 
A.W.~acknowledges the support of ERC 2020 grant HomDyn (grant no.~833423) and SNF grant 217944.
This material is based upon
work supported by a grant from the Institute for Advanced Study School of Mathematics.
}

\date{\today.}
    
    \vspace{-0.5cm}
    

    \begin{abstract}
        Let $Q=-x_1^2-x_2^2-x_3^2+x_4^2$ be the standard signature $(1,3)$ quadratic form.
        To each non-degenerate rational plane $L$ in the four-dimensional quadratic space $(\mathbb{Q}^4,Q)$ we can naturally attach a periodic geodesic on the Bianchi orbifold $\mathrm{SL}_2(\mathbb{Z}[i])\backslash \mathbb{H}^3$ which records the position of $L$ in the Grassmannian up to integer rotations.
        Moreover, each such plane $L$ defines a CM point and a periodic geodesic on the modular curve through restriction of $Q$ to $L$ and its orthogonal complement.
        Lastly, the local isomorphism between $\mathrm{SO}_{1,3}(\mathbb{R})$ and $\mathrm{SL}_2(\mathbb{C})$ gives rise to a further periodic geodesic on the Bianchi orbifold.
        
        In this article, we exhibit a natural coupling of all the above objects and prove simultaneous equidistribution under a Linnik-type splitting condition.
        The main ingredient is the classification of joinings of higher-rank diagonalizable actions on homogeneous spaces due to Einsiedler and Lindenstrauss.
    \end{abstract}

	\maketitle


    \section{Introduction}\label{sec:introduction}
    
    Duke's theorem and its extensions form a family of equidistribution results describing, for example, the distribution of periodic geodesics or complex multiplication (CM) points on the modular surface, and of integer points on large spheres. Distribution properties of these objects were already studied by Linnik and his school \cite{Linnik1968ErgodicPO}.
    Recent advances in our understanding of higher-rank diagonalizable actions on homogeneous spaces due to Einsiedler and Lindenstrauss~\cite{einsiedler_joinings_2019} opened the door to the study of natural coupled-distribution problems of the above objects, or other arithmetic objects of interest \cite{AkaEinsiedlerShapira-1in3,aka_planes_2021,ALMW,BlomerBrumley}.
    Most prominently, Khayutin~\cite{Khayutin-mixing} used these advances to make significant progress towards the mixing conjecture of Michel and Venkatesh~\cite{MichelVenkateshICM}.
    In later developments, analytic methods \textemdash~see the work of Blomer, Brumley~\cite{BlomerBrumley}, Blomer, Brumley, Khayutin~\cite{BlomerBrumleyKhayutin}, and Blomer, Brumley, {Radiwi\l\l} \cite{BlomerBrumleyRadiwill} \textemdash~have established important equidistribution results in this context under the increasingly weaker assumptions on zeros of certain $L$-functions.
    
    This article studies two-dimensional rational subspaces in $\QQ^4$ equipped with a quaternary quadratic form, which gives rise to a notion of discriminant on these subspaces, and to various natural objects associated to these subspaces. In a precursor of this work, two of us (M.A.~and A.W.) with Einsiedler~\cite{aka_planes_2021} studied two-dimensional rational subspaces in $\QQ^4$ equipped with the quadratic form $x_1^2+x_2^2+x_3^2+x_4^2$. In this case,
    to each two-dimensional rational (oriented) subspace $L$ (henceforth called a \emph{plane}) one can associate four complex multiplication points (CM points) in the modular curve $\SL_2(\ZZ)\backslash \HypSpace^2$. 
    Namely, two of these CM points describe the shape of the lattice $L \cap \ZZ^4$ in $L$ and its orthogonal complement $L^\perp \cap \ZZ^4$ in $L^\perp$. 
    Equivalently, they are (the equivalence classes of) the binary forms obtained by restricting $x_1^2+x_2^2+x_3^2+x_4^2$ to $L$ and $L^\perp$.
    The other two CM points arise from an accidental local isomorphism between $\SO_4$ and $\SO_3 \times \SO_3$.
    Geometrically, this isomorphism allows one to view, very explicitly, the affine variety of oriented two-dimensional subspaces as a product of two $2$-dimensional spheres. Under this isomorphism, each rational plane is parametrized by two vectors in $\ZZ^3$ in the ternary quadratic space $(\QQ^3,x^2+y^2+z^2)$, called Klein vectors in \cite{aka_planes_2021}.
    In particular, by taking the orthogonal complement of the Klein vectors, two further definite binary forms are associated to a rational plane.
    Conjecturally, the \enquote{coupled} tuples consisting of subspaces and their four associated CM points equidistribute in the appropriate product space (a Grassmannian times four copies of the modular curve) when the subspaces are varied with fixed discriminant and the discriminant goes to infinity.
    In \cite{aka_planes_2021}, a slightly weaker version of this is established using \cite{einsiedler_joinings_2019}.
    
    It was hinted at in \cite{aka_planes_2021} that these constructions can be extended to any quaternary space arising from the norm form on rational quaternion algebras. Indeed, in later work with Feller and Miller \cite{AFMW24}, M.A.~and A.W.~used this observation for the quadratic form $x_1x_4-x_2x_3$ arising from the rational quaternion algebra of rational $2\times 2$ matrices toward an application in low-dimensional topology. 
    
    The aim of this article is to study rational quadratic forms that do not arise from a rational quaternion algebra, e.g. the signature $(1,3)$ form $\Q = -x_1^2-x_2^2-x_3^2+x_4^2$ on $\QQ^4$. 
    There are two main novelties. 
    The first is the study of the distribution problem of rational planes $L$ and the binary forms arising from restricting $\Q$ to $L$ and $L^\perp$. 
    This already presents an interesting coupled distribution problem of classical arithmetic objects. 
    Explicitly, as explained below, it is a coupling between geodesics on the Bianchi orbifold, geodesics in the modular surface and CM points. 
    Already the individual equidistribution of these associated arithmetic objects is of interest, see e.g.~Theorem \ref{thm:PositionEquiIntro}. 
    The second novelty is a generalization of the Klein vector construction from \cite{aka_planes_2021} via Clifford algebras, which allows us to treat general quaternary quadratic spaces. 
    For $\Q$, this construction gives rise to an additional coupling with a geodesic in the Bianchi orbifold.
    
    We turn to describing our main results more precisely. For simplicity of exposition and for technical reasons, we limit ourselves to the quadratic space $(\QQ^4,\Q)$ where, as above, $\Q = -x_1^2-x_2^2-x_3^2+x_4^2$.
    It is useful to view this quadratic form as the determinant form on the subspace $\mathcal{H}$ of Hermitian matrices in $\Mat[2](\QQi)$ together with the coordinates
    \begin{align*}
    \begin{pmatrix}
    -x_3- x_4 & x_2 + ix_1 \\
    x_2 - ix_1 & x_3-x_4
    \end{pmatrix}.
    \end{align*}
    Note that $\SL_2(\QQi)$ acts on $\mathcal{H}$ via $g.x = gx g^{\mathsf{H}}$, where $g^{\mathsf{H}}$ denotes the hermitian involution (transpose conjugate) of $g\in \SL_2(\QQi)$.
    This action preserves the determinant. 
    Thus, this shows that $\SO_Q$ is, in fact, isogenous to $\Res{\QQi}(\SL_2)$ which is an \emph{outer} form of $\SL_2\times \SL_2$. This contrasts the situation here with the setting in~\cite{aka_planes_2021}.
    
    Consider now an oriented two-dimensional subspace $L \subset \QQ^4$.
    Its discriminant (with respect to $\Q$) \textemdash~denoted by $\disc(L)$ \textemdash~is the discriminant of the restriction of $\Q$ to the two-dimensional lattice $L \cap \ZZ^4$.
    We will simply say that $L$ is definite, indefinite etc, if $Q|_L$ is definite indefinite etc.
    The discriminant $\disc(L)$ is positive if $L$ is definite, negative if $L$ is indefinite, and zero if $L$ is degenerate.
    Since $Q$ has signature $(1,3)$, any definite two-dimensional subspace is negative definite.  
    For $D \geq 1$ we set
    \begin{align}\label{eq:Def-of-JD}
    \P_D = \{L: \disc(L) =D\}.
    \end{align}
    Note that $\P_D$ is acted upon by $\SO_Q(\ZZ)$ and, in particular, also by the infinite group $\SpinV(\ZZ)\simeq \SL_2(\ZZi)$. 
    An elementary argument shows that $\P_D$ is non-empty for every $D \geq 1$.
    Moreover, by finiteness of class numbers the set of orbits $\J_D = \SpinV(\ZZ) \backslash \P_D$ is finite.
    
    Suppose $\L$ is definite (if $L$ is indefinite, consider its orthogonal complement). 
    As explained above, we will associate to it several arithmetic objects. We start with the \enquote{position} of $\L$.  
    The \enquote{position} of $\L$ is contained in a proper open subset of the real affine variety of oriented two-dimensional real subspaces (as opposed to the case considered in \cite{aka_planes_2021}) and the stabilizer of a point under the $\SO_\Q(\RR)$-action is non-compact; this is an unfavorable situation. 
    Instead, by \enquote{dualizing} our setting, the position of $\L$ is captured by a periodic geodesic in the unit tangent bundle $\Tan(\Y_{\QQi})$ of the Bianchi orbifold
    \begin{align}\label{eq:Bianchiorbifold}
        \Y_{\QQi} \defeq \dquot{\SL_2(\ZZi)}{\SL_2(\CC)}{\SU_2(\CC)} = \lquot{\SL_2(\ZZi)}{\HypSpace^3}
    \end{align}
    associated to the ring of integers $\ZZi$. The base point of this geodesic is essentially given by any element conjugating the stabilizer of $\L$ into the full diagonal subgroup. Indeed, note that $\Y_{\QQi} \simeq \tlquot{\SL_2(\ZZi)}{\HypSpace^3}$ by acting with $\SL_2(\CC)$ on $\HypSpace^3$ by extending the Möbius transformation on its boundary $\partial\HypSpace^3 \simeq \CC$. The group $\SU_2(\CC)$ is the stabilizer of $o\defeq(0,1) \in\HypSpace^3 \simeq\partial\HypSpace^3\times\RR_{>0}\simeq\CC\times\RR_{>0}$. 
    Thus, the unit tangent bundle can be written as
    \[ \Tan(\Y_{\QQi}) \defeq \dquot{\SL_2(\ZZi)}{\SL_2(\CC)}{A_\CC^1}, \]
    considering that $A_\CC^1 = \{\operatorname{diag}(e^{it},e^{-it})\st t\in\RR\}\subseteq\SU_2(\CC)$ is the stabilizer of $(o,v_o)$, where $v_o$ is the unit tangent vector at $o$ \enquote{pointing upwards}. Geodesics in $\Tan(\Y_{\QQi})$ correspond to $A=\{\operatorname{diag}(e^{t},e^{-t})\st t\in\RR\}$-orbits. Moreover, notice that $AA_\CC^1 = A_\CC=\{\operatorname{diag}(a,a^{-1})\st a\in\CC\}$. 
    Here, we mention that the length of the periodic geodesic, in fact, only depends on the discriminant $\disc(\L)$ and that the geodesic itself only depends on the class of $\L$ under integral rotations. We refer to \S\ref{sec:shapes} for a completely explicit, albeit technical setup.

    \begin{theorem}[Duke-type theorem for periodic geodesics on a Bianchi orbifold]\label{thm:PositionEquiIntro}
        Let $\nu_D$ be the (finite) normalized sum of the length measures on all periodic geodesics associated to classes of subspaces in $\J_D$.
        Then $\nu_D$ converges to the uniform measure on $\Tan(\Y_{\QQi})$ as $D \to \infty$ along non-squares.
    \end{theorem}
    
    This theorem relates to a variant of a Linnik-type equidistribution problem over the quadratic field $\QQi$; see~\S\ref{subsec:EquiEachFac}.
    
    We will now \enquote{couple} the measure $\nu_D$ considered in Theorem \ref{thm:PositionEquiIntro} to other measures of similar kind using geometric constructions.
    So let $L$ be, again, an oriented two-dimensional definite subspace of $\QQ^4$. 
    The restriction of $\Q$ to the integer lattice $L(\ZZ) = L \cap \ZZ^4$ in $L$ is a definite binary form uniquely determined up to the standard $\SL_2(\ZZ)$-action (as the subspace is oriented). We thus obtain a point in the modular surface
    \begin{align*}
        [\LZ] \in \Y \defeq \dquot{\SL_2(\ZZ)}{\SL_2(\RR)}{\SO_2(\RR)} \simeq \lquot{\SL_2(\ZZ)}{\HypSpace^2}
    \end{align*}
    Restricting $\Q$ to the orthogonal complement $L^\perp$ (we will equip it below with a natural orientation), we get an $\SL_2(\ZZ)$-class of binary indefinite quadratic forms. Such a class gives rise to a periodic geodesic on the unit tangent bundle of the modular surface - the space of two-dimensional unimodular $\ZZ$-lattices -
    \[ \Tan(\Y) \defeq \lquot{\SL_2(\ZZ)}{\SL_2(\RR)}. \]
    The length of this geodesic is closely related to the length of the geodesic we viewed as the \enquote{position} of $L$ (cf.~Lemma~\ref{lem:lengths}). Versions of these constructions exist in any dimension (e.g.~as in \cite{aka_equidistribution_2021}). 
    
    We now turn to construct, given $L$, another geodesic on the above Bianchi orbifold, generalizing the Klein vectors construction mentioned above. This is a specialization of a more abstract construction, which readily works for any quaternary rational quadratic form and is carried out in \S\ref{sec:geodesic-qoKL} using Clifford algebras. It is the sought-after generalization of the Klein vectors construction mentioned above.
    
    Recall that we identify $\QQ^4$ with the set of Hermitian matrices $\mathcal{H} \subset \Mat[2](\QQi)$. Given an oriented integer basis $v_1,v_2$ of $L$, define
    \begin{align*}
    \Klein(\L) \defeq v_1\adj[v_2] - v_2\adj[v_1]
    \end{align*}
    where $\adj[\cdot]$ denotes the adjugate.
    The vector $\Klein(L)$, henceforth called Klein vector, is thus a $\ZZi$-integral matrix independent of the choice of basis. Also, one can check that $\det(\Klein(L))=4\disc(L)$. Take the orthogonal complement of $\Klein(L)$ within traceless matrices over $\QQi$ and with respect to the quadratic form $\det$. The quadratic form $\det$ on the $\ZZi$-lattice within this orthogonal complement gives rise to an $\SL_2(\ZZi)$-equivalence class of a binary quadratic form over $\ZZi$. 
    Such a class gives rise to a periodic geodesic on the unit tangent bundle of the Bianchi orbifold $\Y_{\QQi}$.
    
    The geometric constructions associated with a plane $\L$ dictate a specific intertwining of the above objects, as we now describe.
    For any oriented definite rational plane $\L$ one can define a probability measure on $\Tan(\Y_{\QQi})\times\Y\times\Tan(\Y)\times\Tan(\Y_{\QQi})$ by
    \begin{align*}
        \nu^{joint}_{\L}(\varphi) = \frac{1}{\ell_D}\int_{0}^{\ell_{D}}\!\varphi(a_s.z_{\L,1},z_{\L,2},a_s^2.z_{\L,3},a_s^2.z_{\L,4})\dd{s}.
    \end{align*}
    for all continuous compactly supported functions $\varphi$.
    Here, $\ell_{D}$ denotes the length of the periodic geodesic in the first component, which is determined by $\L$ and only depends on its discriminant $D$, $a_s$ denotes the unit-speed geodesic flows in the corresponding spaces, and $z_{\L} \defeq (z_{\L,1},z_{\L,2},z_{\L,3},z_{\L,4})$ is an interdependent tuple of points lying on the corresponding objects (determined by the geometric constructions above). 
    We refer to \S\ref{sec:planes-to-equi} for an explicit definition.
    The measure $\nu^{joint}_\L$ turns out to be only dependent on the $\SpinV(\ZZ)$-class of $\L$ and we define
    \begin{align*}
    \nu^{joint}_D = \frac{1}{|\J_D|} \sum_{\SpinV(\ZZ)\L \in \J_D} \nu^{joint}_\L.
    \end{align*}
    The measure $\nu^{joint}_D$ projects precisely to $\nu_D$ in the first factor $\Tan(\Y_{\QQi})$ and to similar measures in the remaining factors.

    The following is our main result.

    \begin{theorem}\label{thm:equi}
        Let $p$ be an odd prime. As $D\to\infty$ with $D$ square-free, odd and $-D$ or $D$ a non-zero square modulo $p$, the measures $\nu^{joint}_D$ converge to the uniform probability measure on 
        \[  \Tan(\Y_{\QQi})\times\Y\times\Tan(\Y)\times\Tan(\Y_{\QQi}). \]
    \end{theorem}

    We note that, conjecturally, one should be able to omit the congruence condition at the auxiliary prime $p$. 
    Removing the auxiliary splitting prime requires significant efforts and, in particular, necessitates avoiding the aforementioned rigidity results of Einsiedler and Lindenstrauss. These rigidity results require congruence conditions at two distinct places.
    Using a more elementary disjointness argument (see \S\ref{sec:proof-dyn}), we can make use of the prime at infinity, that is, the geodesic flow.

    \begin{remark}[Other quadratic forms]
    It is likely that a theorem of the above nature can be established for general quadratic forms; in particular, we extend the Klein vectors to this generality. The distribution results in this article are proved for the form $-x_1^2-x_2^2-x_3^2+x_4^2$ for simplicity of exposition and as this choice relates to natural arithmetic objects. Additionally, it allows us to avoid technical difficulties pertaining to the Steinitz class of the complement of the Klein vector.
    \end{remark}
    
    \begin{remark}[Other future directions]
    To the authors' knowledge, analytic methods \cite{BlomerBrumley,BlomerBrumleyKhayutin,BlomerBrumleyRadiwill} have not yet provided any progress on an equidistribution problem in more than $2$ factors as in Theorem~\ref{thm:equi}; such progress would doubtlessly be interesting to the community.
    
    In a precursor of the problem studied here, Aka, Einsiedler, Shapira \cite{AkaEinsiedlerShapira-1in3,AkaEinsiedlerShapira-grids} considered integer points on spheres together with their orthogonal lattices. That slightly simpler problem can be refined to consider additionally a natural marked point on the torus associated to each CM point \textemdash~see \cite{AkaEinsiedlerShapira-grids} and, particularly, the striking work of Khayutin \cite{Khayutin-grids} relying on geometric invariant theory. The authors are unaware of a similar refinement in the context of this paper or of \cite{aka_planes_2021}.
    \end{remark}
    	
    \subsection{Outline of the paper}\label{sec:outline}
    
    This paper is organized as follows.
    In \S\ref{sec:Clifford-algebra-and-Klein-map}, we generalize the construction of Klein vectors (which vaguely parametrizes two-dimensional subspaces) from \cite{aka_planes_2021} to general quadratic forms. Here, the main tools include the notion of Clifford algebra for a quadratic form.
    In \S\ref{sec:shapes}, we discuss the arithmetic objects and quadratic forms attached to each non-degenerate oriented rational plane. 
    In \S\ref{sec:dynamical-theorem}, we formulate our dynamical main theorem, \Cref{thm:dynamical-version}, and use it to deduce \Cref{thm:equi} above. Finally, in \S\ref{sec:proof-of-dyn-thm} we prove the dynamical theorem (\Cref{thm:dynamical-version}) using the aforementioned classification of joinings of higher-rank actions of Einsiedler and Lindenstrauss \cite{einsiedler_joinings_2019} and a variant of Duke's equidistribution theorem \cite{duke_hyperbolic_1988} (cf.~Theorem~\ref{thm:Duke}).
    
    \subsection{Notation}
    
    For a locally compact group $G$ and $\Gamma < G$ a lattice, the quotient $\tlquot{\Gamma}{G}$ is equipped with a (left) $G$-action $g.\Gamma h = \Gamma hg^{-1}$. We denote by $m_{\Gamma\backslash G}$ the $G$-invariant probability measure on $\tlquot{\Gamma}{G}$.
    
    Let $\Adeles$ (respectively $\Adeles_f$) be the ring of adeles (respectively finite adeles).
    We view $\QQ$ as diagonally embedded in $\Adeles$ and, as such, $\QQ$ is discrete.

    Given a semisimple linear algebraic group $\mathbf{M}$ defined over $\QQ$, the 
    subgroup $\mathbf{M}(\QQ) < \mathbf{M}(\Adeles)$ is a lattice by a theorem of Borel and Harish-Chandra \cite{borel_properties_1961}.
    For $g\in\mathbf{M}(\Adeles)= \mathbf{M}(\R) \times \prod_p' \mathbf{M}(\QQ_p)$, we denote its components by $g = (g_\infty,g_2,g_3,\dots)$.
    We say that $\mathbf{M}$ has class number one with respect to a compact open subgroup $K_f < \mathbf{M}(\Adeles_f)$, if
    \begin{align}\label{eq:class number one}
    \mathbf{M}(\Adeles) = \mathbf{M}(\QQ)(\mathbf{M}(\RR)\times K_f).
    \end{align}
    If $\mathbf{M}$ is simply connected and $\mathbf{M}(\RR)$ has no compact factors, strong approximation implies that $\mathbf{M}$ has class number one with respect to any compact open subgroup of $\mathbf{M}(\Adeles_f)$.

    We let
    \begin{align}\label{eq:defG}
    \GG = \mathrm{Res}_{\QQi/\QQ}(\SL_2) \times \SL_2 \times \SL_2 \times \mathrm{Res}_{\QQi/\QQ}(\SL_2)
    \end{align}
    and denote by $\GG_1,\ldots,\GG_4$ the factors of $\GG$.
    These are simply connected groups and have class number one with respect to any compact open subgroup of the finite adelic points.
    Throughout, we will denote the components of a geometric point $g$ in $\GG$ by $g_j$ in the corresponding factor $\GG_j$ of $\GG$, $1\leq j \leq 4$.
    
    \subsection*{Acknowledgements}
    We are very grateful to Philippe Michel for providing us with an outline of the proof of Theorem~\ref{thm:Duke} (a variant of Duke's theorem) and for directing us to useful references mentioned in \S\ref{subsec:EquiEachFac}.
    We thank Manfred Einsiedler for inspiring discussions. 
    A.W.~would also like to thank the Forschungsinstitut für Mathematik at ETH Zurich and the Institute for Advanced Study for providing such excellent work environments.
	
\section{Clifford Algebras and the Klein map}\label{sec:Clifford-algebra-and-Klein-map}
	
    We will now construct a map \textendash~henceforth called Klein map \textendash~\enquote{parametrizing} two-dimensional subspaces of an arbitrary non-degenerate four-dimensional quadratic vector space $\V$ over $\QQ$. The Klein map is constructed using Clifford algebras into which we give a short introduction in the following. 

\subsection{Preliminaries on Clifford Algebras}\label{sec:prel-clifford-alg}

    Here, we recall \textendash~mostly without proofs \textendash~a few standard facts about Clifford algebras. We refer to \cite{Knus-campinas}, \cite{cassels2008rational}, or \cite[Chapter IV]{knus_quadratic_1991} for a more thorough discussion. 
    
    Throughout the following, $\KK$ will always denote a general field of characteristic zero. We consider a $d$-dimensional quadratic space $(\V,\Q)$ over $\KK$, i.e.~a vector space $\V$ over $\KK$ equipped with a quadratic form $\Q\colon\V\to\KK$. 
    We write $\bQ$ for the bilinear form associated to $\Q$ given by 
    \begin{align}\label{eq:bilinear-form}
        \bQ(x,y) = \Q(x+y)-\Q(x)-\Q(y).
    \end{align} 
    We call the quadratic space $(\V,\Q)$ \highlight{non-degenerate} if the bilinear form $\bQ$ is a perfect pairing, i.e.~if the adjoint ${\bQ}^*\colon\V\to\V^*,~x\mapsto \bQ(\cdot,x)$, of $\bQ$ is an isomorphism (the terms \enquote{non-singular} or \enquote{regular} can also be found in the literature). 
    The \highlight{discriminant} of $\Q$ with respect to a basis $\{e_1,\dots,e_d\}$ is
    \[ \disc(\Q) \defeq \det((\bQ(e_i,e_j)/2)_{i,j}). \]
    Changing the basis alters the discriminant by a non-zero square. We note that $(\V,\Q)$ is non-degenerate exactly when the discriminant is non-zero.

    \begin{definition}\label{def:clifford-alg}
        The \highlight{Clifford algebra} of the quadratic space $(\V,\Q)$ is a $\KK$-algebra $\Cl$ together with a $\KK$-linear map $\iota\colon\V\to\Cl$,  which is universal with respect to the property
    	\begin{align}\label{eq:universal-property}
    		\iota(x)^2 = \Q(x)
    	\end{align}
    	for all $x\in\V$: if $A$ is any $\KK$-algebra with a $\KK$-linear map $\phi: \V \to A$ satisfying $\phi(x)^2 = \Q(x)$ for all $x\in\V$, there is a $\KK$-algebra homomorphism $\phi': \Cl\to 
        A$ such that $\phi' \circ \iota = \phi$.
    \end{definition}
    
    To construct the Clifford algebra explicitly, take the graded tensor algebra \linebreak$T(\V)=\KK\oplus\V\oplus(\V\otimes\V)\oplus\cdots$ and the quotient of it by the two-sided ideal $\mathcal{J}$ of $T(\V)$ generated by all elements of the form $x\otimes x - \Q(x)$ for $x \in \V$. 
    That is,
	\[ \Cl \defeq \quot{T(\V)}{\mathcal{J}} \]
	together with the canonical map $\iota\colon\V\to\tquot{T(\V)}{\mathcal{J}}$. 
    
    Henceforth, using $\iota$, we will view the quadratic space $\V$ as a subset of its Clifford algebra. 
    The Clifford algebra comes with a natural grading into \highlight{even} and \highlight{odd} part
	\[ \Cl = \Cle \oplus \Clo, \]
    where the even and odd parts corresponds exactly to the even and odd degree tensors in the above construction. 
    
	Using the relation \eqref{eq:universal-property} for $x+y$ it follows that
	\begin{align}\label{eq:char-property}
        xy + yx = \bQ(x,y), 
    \end{align}
    so, in particular, orthogonal vectors anti-commute. The dimension of $\Cl$ as a $\KK$-vector space is $2^{d}$. In fact, if $\{e_1,\ldots,e_d\}$ is a basis of $\V$, then
    \begin{align}\label{eq:basis Clifford algebra}
    \{1\}\cup\{e_{i_1}\cdots e_{i_j} \st i_1 < \cdots < i_j, j=1,\ldots,d \}
    \end{align}	
    is a basis of the Clifford algebra.
    The invertible elements of the even Clifford algebra act on the Clifford algebra by conjugation. When convenient, we denote this action by $\Adj_{\bullet}$.

	\begin{definition}
		The \highlight{special Clifford group} is
        \begin{align}\label{eq:special-cl-group}
            \sClGrV(\KK) = \{\alpha \in\Cle^\times: \alpha\V\alpha^{-1} \subset \V\}.
        \end{align}
	\end{definition}
    The action of $\sClGrV(\KK)$ on $\V$ by conjugation yields a short exact sequence
    \begin{align}1\longrightarrow\KK^\times
        \longrightarrow&\sClGrV(\KK)\longrightarrow\SO_\Q(\KK)\longrightarrow1.\label{eq:ses-of-sClGr}
	\end{align}

\begin{remark}
    This is a consequence of the well-known fact that any element of $\SO_{\Q}(\KK)$ can be written as a product of an \emph{even} number of reflections. Surjectivity thus follows from the relation $-vwv^{-1} = w - \frac{\bQ(v,w)}{\Q(v)}v$ for anisotropic vectors $v$.
\end{remark}

    We define the \highlight{standard involution} $\stdinv\colon\Cl\to\Cl$ by linearly extending the map $v_1\cdots v_m \mapsto (-v_m)\cdots (-v_1)$ for all $v_1,\dots,v_m\in\V$. It follows that $\stdinv(\alpha\beta) = \stdinv(\beta)\stdinv(\alpha)$ for all $\alpha,\beta\in\Cl$. The standard involution gives rise to the \highlight{norm form}
	\begin{align*}
		\ClNrm(\alpha) &\defeq \alpha\stdinv(\alpha), \quad \alpha\in\Cl.
	\end{align*}
	For $x\in\V$ we have $\ClNrm(x) = -\Q(x)\in\KK$, but $\ClNrm$ does not map $\Cl$ to~$\KK$ in general.
    
    \begin{lemma}\label{lem:norm-induces-hom-on-ClGr}
		The norm form $\ClNrm$ induces a group homomorphism $\ClNrm\colon\sClGrV(\KK)\to\KK^\times$.
	\end{lemma}

	The \highlight{spin group} of $(\V,\Q)$ is, by definition, the kernel of the above map
	\begin{align}\label{eq:spin-group}
		\SpinV(\KK) &\defeq \{\alpha\in\sClGrV(\KK)\st \ClNrm(\alpha) = 1 \}.
	\end{align}
    Taking the norm of a lift of an element in $\SO_\Q$ via \eqref{eq:ses-of-sClGr}, we obtain the last map in the short exact sequence
    \begin{align}\label{eq:spin-exact-seq}
        1 \to \{\pm 1\} \to \SpinV(\KK) \to \SO_\Q(\KK) \to \KK^\times/(\KK^\times)^2.
    \end{align}

    As indicated by the definition in \eqref{eq:spin-group}, $\SpinV(\KK)$ is the set of $\KK$-points of a linear algebraic group $\SpinV$ defined over $\KK$.
    For a field $\mathbb{L}$ containing $\KK$, the quadratic form $\Q$ extends uniquely to a quadratic form $\V\otimes_{\KK}\mathbb{L} \to \mathbb{L}$ which we again denote by $\Q$ for simplicity.
    It holds that $\Cl[(\V \otimes_{\KK} \mathbb{L},\Q)] \simeq \Cl\otimes_{\KK}\mathbb{L}$ and so $\SpinV(\mathbb{L})$ is the set of $\mathbb{L}$-points of the linear algebraic group $\SpinV$ defined over $\KK$.
    The analogous statement is true for the special Clifford group.
    Linear equations to define $\SpinV$ can be obtained from \eqref{eq:special-cl-group} and \eqref{eq:spin-group}. 
    The spin group $\SpinV$ is a simply connected semisimple group and an algebraic double covering of $\SO_\Q$ (by \eqref{eq:spin-exact-seq}).

\subsection{Quaternary Quadratic forms and the Klein map}\label{sec:klein-map}
    From now on, we assume that $\V$ has dimension $4$. 
    By computing $\stdinv$ on the basis in \eqref{eq:basis Clifford algebra} using an orthogonal basis of $\V$, we have	
    \begin{align}\label{eq:char-of-V}
		\V = \{v\in\Clo\st \stdinv(v) = -v\}.
	\end{align}

    This characterization of $\V$ has the following implication.
    \begin{lemma}\label{lem:norm-implies-ClGr}
		For any $\alpha\in\Cle$ with $\ClNrm(\alpha)\in\KK^\times$ we have $\alpha\in\sClGrV(\KK)$. In particular,
        \begin{align} \label{eq:spinV}
            \SpinV(\KK) = \{\alpha \in \Cle: \ClNrm(\alpha)=1 \}.
        \end{align}
	\end{lemma}
	\begin{proof}
        Note that $\ClNrm(\alpha)\in\KK^\times$ implies $\alpha^{-1} = \tfrac{\stdinv(\alpha)}{\ClNrm(\alpha)}$. 
        Hence, for all $v\in\V$ we have
        \[ \stdinv(\alpha v \alpha^{-1}) 
        = \tfrac{1}{\ClNrm(\alpha)}\alpha\stdinv(v)\stdinv(\alpha) = -\tfrac{1}{\ClNrm(\alpha)}\alpha v\stdinv(\alpha) = -\alpha v \alpha^{-1} \]
        and using \eqref{eq:char-of-V} we conclude $\alpha \V \alpha^{-1} \subseteq \V$ and thus $\alpha\in\sClGrV(\KK)$.
	\end{proof}

    We denote the center of the even Clifford algebra by $\Z \defeq \Z(\Cle)$,
    which is a so-called separable quadratic $\KK$-algebra. Explicitly, if $\{e_1,e_2,e_3,e_4\}$ is an orthogonal basis of $\V$, we have that $\{1, e_1e_2e_3e_4\}$ is a $\KK$-basis of $\Z$. By orthogonality and the relations \eqref{eq:char-property}, $(e_1e_2e_3e_4)^2$ is equal to $\disc(\Q)$. If $\disc(\Q)\not\in(\KK^\times)^2$, then $\Z$ is the quadratic extension $\KK[x]/(x^2-\disc(Q))$ of $\KK$.
    If $\disc(\Q)\in(\KK^\times)^2$, then $\Z\cong\KK\times\KK$.
    The quadratic $\KK$-algebra $\Z$ comes with a unique involution (which is the Galois automorphism if $\Z$ is a field, respectively flips the coordinates if $\Z\simeq\KK\times\KK$) giving rise to a trace form $\trZ$ and a norm form $\nrZ$.
	\begin{proposition}\label{prop:structure-of-Cle}
		Let $(\V,\Q)$ be a non-degenerate quaternary quadratic $\KK$-vector space. If $\Z$ is a field, then $\Cle$ is a quaternion algebra over $\Z$, and if $\Z \simeq \KK\times\KK$, then $\Cle\simeq A\times A$ where $A$ is a quaternion algebra over $\KK$.
        Moreover, $\stdinv$ corresponds to the standard involution of the quaternion algebra in the former case and the product of the standard involutions in the latter case.
	\end{proposition}

    The proposition and its generalization to arbitrary dimension are completely standard, see e.g.~\cite[p.~45]{Knus-campinas} or \cite[IV. (2.2.3)]{knus_quadratic_1991}; we give a concrete argument for the readers' convenience.
    
    \begin{proof}
        Let $\{e_1,e_2,e_3,e_4\}$ be an orthogonal basis of $\V$. Recall that $\Z$ is spanned by $\{1, e_1e_2e_3e_4\}$ over $\KK$.
        First, assume that $\Z$ is a field. Then $\Cle$ is four-dimensional (as vector space) over $\Z$ and the basis $\{1, e_1e_2, e_2e_3, \Q(e_2)e_1e_3\}$ satisfies
        \begin{align*}
        (e_1e_2)^2 = -\Q(e_1)\Q(e_2),&\quad
        (e_2e_3)^2 = -\Q(e_2)\Q(e_3),\\
        (e_1e_2)(e_2e_3) = \Q(e_2)&e_1e_3 =-(e_2e_3)(e_1e_2).
        \end{align*}
        Thus, $\Cle$ is a quaternion algebra over $\Z$.

        Now assume that $\Z\simeq\KK\times\KK$. 
        Recall that $\disc(Q)$ is a square in $\KK$ in this case and write $x^2 = \disc(\Q)$ for some $x\in \KK^\times$. 
        Idempotents in $\Z$ are then given by
        \begin{align*}
            z_1 = \tfrac{1}{2}\left(1 + \frac{e_1e_2e_3e_4}{x}\right) \simeq (1,0),\quad z_2 = \tfrac{1}{2}\left(1 - \frac{e_1e_2e_3e_4}{x}\right)\simeq (0,1)
        \end{align*}
        and $z_1z_2 = z_2z_1 = 0$ as well as $z_1+z_2=1$. 
        In particular, we can decompose every $z\in\Cle$ as $z = z_1z + z_2z$.
        One can check that $z_1v + vz_1 = v$ and $vz_1 = z_2v$ holds for all $v\in\V$.
        As in the first case, the span of $\{z_1, z_1e_1e_2, z_1e_2e_3, \Q(e_2)z_1e_1e_3\}$ over $\KK$ is a quaternion algebra over $\KK$ and is isomorphic to the analogously defined quaternion algebra using the idempotent $z_2$. The above decomposition of $\Cle$ thus implies the proposition in this case.

        It is straightforward to check that $\stdinv$ corresponds to the standard involutions or the product of standard involutions, respectively.
    \end{proof}

\subsubsection{Klein map}\label{sec:Klein-map}
    Let $(\V,\Q)$ be a non-degenerate quaternary quadratic space over $\QQ$ with a choice of an orientation on $\V$ as a real vector space.
    In this section, we define the Klein map and describe how oriented planes in $\V$ give rise to \emph{Klein vectors} which are traceless vectors in the even Clifford algebra.

    We denote by $\varV[]$ the variety of pure $2$-wedges in the affine four-space, so that 
    \[ \varV[](\QQ) 
    = \big\{v_1\wedge v_2\in \textstyle{\bigwedge^2}\V: v_1,v_2\in\V\big\} 
    =\big\{w \in \bigwedge^2\V: w \wedge w =0\big\}. \]
    For a pure $2$-wedge $v_1\wedge v_2\in\varV[]$ we define its discriminant (with respect to $\Q$) as
    \[ \disc(v_1\wedge v_2) \defeq \det{\smat{\Q(v_1)}{\bQ(v_1,v_2)/2}{\bQ(v_2,v_1)/2}{\Q(v_2)}}, \]
    and denote by $\varV[nd] \defeq \{v_1\wedge v_2\in\varV[]\st\disc(v_1\wedge v_2)\neq0\}$ the open set of \highlight{non-degenerate} pure $2$-wedges.
    Note that the action of $\sClGrV(\QQ)$ on $\V$ naturally induces an action of $\sClGrV$ on $\varV[]$ via
    \[ \Adj_\alpha(v_1\wedge v_2) = \Adj_\alpha(v_1)\wedge\Adj_\alpha(v_2). \]
        
    We call a non-degenerate pure wedge $v_3\wedge v_4$ \highlight{orthogonal} to another non-degenerate pure wedge $v_1\wedge v_2$ if the plane spanned by $v_3$ and $v_4$ is orthogonal to the plane spanned by $v_1$ and $v_2$. 
    
    We naturally extend the quadratic form $\Q$ to $\V(\KK) \defeq \V\otimes\KK$ and define
    \[ \calK(\KK) \defeq \big\{\alpha\in\Cle[(\V(\KK),\Q)]^\times \st \stdinv(\alpha) = -\alpha,~ \ClNrm(\alpha) \in \KK^\times \big\}, \]
    that is, $\calK(\KK)$ is the set of traceless elements in the even Clifford algebra, which have non-vanishing norm in $\KK$. By \Cref{lem:norm-implies-ClGr}, we know that $\calK(\KK)$ is a subset of $\sClGrV(\KK)$. Moreover, observe that for all $\alpha\in\calK(\KK)$ we have $\ClNrm(\alpha) = -\alpha^2\in\KK^\times$.
    As the notation suggests, $\calK(\KK)$ is the set of $\KK$-points of a quasi-affine variety $\calK$.
    
    \begin{proposition}[Klein Map]\label{prop:Klein-map}
		The map $\Klein\colon\varV[nd](\KK)\to\calK(\KK)$, defined by
		\begin{align} 
			v_1\wedge v_2\in\varV[nd](\KK) \mapsto \Klein(v_1\wedge v_2)\defeq [v_1,v_2] \in\calK(\KK),  \label{def:klein-map}
		\end{align} 
		is a well-defined bijection and is equivariant with respect to the respective $\sClGrV(\KK)$-actions.
        Moreover, if $v_3\wedge v_4$ is orthogonal to $v_1\wedge v_2$ then
        \begin{align}\label{eq:orth-klein}
            \Klein(v_3\wedge v_4) = z\Klein(v_1\wedge v_2) 
        \end{align}
        for some non-zero traceless element $z\in\Z$.
	\end{proposition}
	\begin{proof}
        We may also consider $\Klein$ as a linear map $\Klein\colon\bigwedge^2\V\to\Cle$ given by extending $v_1\wedge v_2 \mapsto [v_1,v_2]$ linearly. Notice that $[\cdot,\cdot]$ is bilinear and antisymmetric so that $\Klein$ in this linear viewpoint is well-defined as a map.
        We have
		\[ \stdinv([v_1,v_2]) = \stdinv(v_1v_2) - \stdinv(v_2v_1) = v_2v_1 - v_1v_2 = -[v_1,v_2], \]
        so that in fact $\Klein\colon\bigwedge^2\V\to\{\alpha\in \Cle: \stdinv(\alpha)=-\alpha\}$.
        Moreover, for any pure $2$-wedge $v_1\wedge v_2\in\varV[](\KK)$  we have using \eqref{eq:char-property}
		\begin{align}\label{eq:norm-Klein-vector}
			\ClNrm(\Klein(v_1\wedge v_2)) &= [v_1,v_2]\stdinv([v_1,v_2]) = -[v_1,v_2]^2 \nonumber\\
            &= (2\Q(v_1)\Q(v_2) - v_1v_2v_1v_2 - v_2v_1v_2v_1) \nonumber\\
			&= 4\Q(v_1)\Q(v_2) - \bQ(v_1,v_2)(v_1v_2 + v_2v_1)\nonumber\\ 
			&= 4\Q(v_1)\Q(v_2) - \bQ(v_1,v_2)^2 \nonumber\\
            &= 4\disc(v_1\wedge v_2) \in\KK. 
		\end{align}
        For non-degenerate pure $2$-wedges $v_1\wedge v_2$ it follows that $\ClNrm(\Klein(v_1\wedge v_2))\in \KK^\times$. This shows that $\Klein$ is well-defined as map $\varV[nd](\KK)\to\calK(\KK)$.
		
        Equivariance follows as $\Adj_\alpha(x\wedge y) = \Adj_{\alpha}(x)\wedge\Adj_{\alpha}(y)$ (by definition) and
		\[ \Adj_\alpha([x,y]) = [\Adj_\alpha(x),\Adj_\alpha(y)] \]
        for all $\alpha\in\sClGrV(\KK)$ and $x,y\in\Cl[(\V(\KK),\Q)]$.

        We now prove that \eqref{def:klein-map} is a bijection. It is the restriction of
        \[ \Klein: \textstyle{\bigwedge^2}\V\to\{\alpha\in \Cle: \stdinv(\alpha)=-\alpha\}, \] which is a bijection. Indeed, it is linear and maps a basis to a basis, since for an orthogonal basis $\{e_1,e_2,e_3,e_4\}$ of $\V$ we have $\Klein(e_i\wedge e_j) = 2e_ie_j$. It remains then to show that $\Klein^{-1}(\alpha)\in\varV[nd](\KK)$ for any $\alpha\in\calK(\KK)$. 
        
        So let $\alpha\in\calK(\KK)$ and note that since $\alpha^2 = -\ClNrm(\alpha)\in\KK^\times$ the action of $\alpha$ on $\V(\KK)$ by conjugation yields an involutory, non-trivial element $\rest{\Adj_{\alpha}}{\V} \in \SO_\Q(\KK)$ (see also \eqref{eq:ses-of-sClGr}). In particular, the only eigenvalues of $\rest{\Adj_{\alpha}}{\V}$ can be $\pm 1$. 
        We claim that both eigenspaces have to be two-dimensional. Clearly, $\rest{\Adj_{\alpha}}{\V}$ is not the identity on $\V(\KK)$ as $\alpha$ is not a scalar. Note that traceless elements of $\Z\otimes\KK$ act as $-\Id_{\V}$ on $\V(\KK)$ as can be checked using an orthogonal basis of $\V(\KK)$. 
        Thus, if $\rest{\Adj_{\alpha}}{\V}$ were equal to $-\Id_{\V}$ then $\rest{\Adj_{z\alpha}}{\V} = \Id_{\V}$ for any non-trivial traceless element $z\in\Z\otimes\KK$ and so in particular, $\alpha \in \Z\otimes\KK$ contradicting that $\stdinv(\alpha)=-\alpha$. Thus, the $(+1)$- and $(-1)$-eigenspaces of $\rest{\Adj_{\alpha}}{\V}$ need to be two-dimensional (and orthogonal to each other).
        
        We define the two-dimensional subspace of $\V(\KK)$ associated to $\alpha$ as
        \[ \L_{\alpha} \defeq \{v\in\V(\KK)\st \Adj_{\alpha}(v) = -v \}. \]
        That is, $\alpha v = -v \alpha$ for $v\in \L_\alpha$. 
        The subspace $\L_{\alpha}$ is non-degenerate because it has an orthogonal complement (namely the $(+1)$-eigenspace of $\rest{\Adj_{\alpha}}{\V}$) and $\Q$ is non-degenerate. If we pick any non-isotropic vector $v\in\L_{\alpha}$, then $\{v, \Q(v)^{-1}v\alpha\}$ is a basis of $\L_{\alpha}$. Indeed, for any $v\in\L_{\alpha}$ we get that $v\alpha$ is in $\V$ by using $\stdinv(v\alpha) = -\stdinv(\alpha)v = \alpha v = -v\alpha$ and \eqref{eq:char-of-V}. We therefore have $\Klein^{-1}(\alpha) = \tfrac{1}{2\Q(v)} v\wedge v\alpha$.

        For the last claim of the proposition, let $v_3\wedge v_4\in\varV[nd](\KK)$ be orthogonal to $v_1\wedge v_2$.
        The above geometric observations imply that the action of $[v_1,v_2]$ on $\V(\KK)$ is given by
        \[ \Adj_{[v_1,v_2]}(v_1) = -v_1\quad \Adj_{[v_1,v_2]}(v_2) = -v_2 \quad  \Adj_{[v_1,v_2]}(v_3) = v_3\quad \Adj_{[v_1,v_2]}(v_4) = v_4 \]
        and similarly for $\Adj_{[v_3,v_4]}$.
        In particular, $[v_1,v_2][v_3,v_4]$ acts as $-\Id_\V$ on $\V(\KK)$ and is hence central, traceless, and non-zero as explained above. Since $[v_1,v_2]^2 = -4\disc(v_1\wedge v_2)$ by \eqref{eq:norm-Klein-vector}, the claim follows for 
        \begin{align}\label{eq:KL-to-KoL}
            z = -\frac{[v_1,v_2][v_3,v_4]}{4\disc(v_1\wedge v_2)}
        \end{align}
        and thus the proposition follows as well.
	\end{proof}

    For any $\nu\in\varV[nd](\KK)$ let us denote the stabilizer subgroup of $\nu$ in $\SpinV$ by $\HH_{\nu}$, that is
    \[ \HH_{\nu} \defeq \{\alpha\in\SpinV \st \Adj_\alpha(\nu) = \nu \}. \]
    We denote the stabilizer subgroup of an element $\beta\in\calK(\KK)$ by
    \[ \HH_\beta \defeq \{\alpha\in\SpinV\st \Adj_\alpha(\beta) = \beta \}. \]
    These are $\KK$-tori in $\SpinV$ of absolute rank $2$ and in particular maximal.
    
    Note that if $\nu'$ is orthogonal to $\nu$, then $\HH_{\nu'} = \HH_{\nu}$.
    Indeed, the spin group acts by orientation-preserving transformations of $\V$ and if the orientation is preserved on the plane corresponding to $\nu$ it also has to be preserved on the complement. 
    Alternatively, this is also a consequence of \Cref{prop:Klein-map}. Indeed, by equivariance we have $\HH_\nu = \HH_{\Klein(\nu)}$ for any $\nu\in\varV[nd](\KK)$.

    Given a non-degenerate plane $ L \subset V$, there is a natural embedding of the Clifford algebra $\Cliff(\L,\Q|_\L)$ into the Clifford algebra $\Cl$ given by the inclusion $L \hookrightarrow V$.
    Under this embedding, the special Clifford group for $L$ is mapped into the special Clifford group for $\V$ as any $\alpha$ in this image preserves $L$ under conjugation (by definition) and preserves $L^\perp$ (as it commutes with it); the same applies to the spin group.
    In the following, the above inclusions are implicit and we write $\sClGr_{\L},\Spin_{\L}$ for the respective groups for simplicity.
 
    \begin{corollary}\label{cor:stabilizer-subgroup}
        If $\L$ is the plane in $\V$ defined by $\nu$ and $\oL$ is its orthogonal complement, then we have
    	\begin{align}\label{eq:char-stabilizer}
    	   \HH_{\nu}(\KK) = \{\beta\beta'\in\SpinV(\KK) \st \beta\in\sClGr_{\L}(\KK), \beta'\in\sClGr_{\oL}\!(\KK) \text{ with } \ClNrm(\beta)\ClNrm(\beta') = 1\}. 
    	\end{align}
        In particular, $\HH_{\nu}$ is isogenous to $\Spin_{\L}\times \Spin_{\oL}$.
        The natural action of an element $\beta\beta'\in\HH_{\nu}(\KK)$ on $\L$ is given by $\ClNrm(\beta')\beta^2 \in \Spin_L(\KK)$, $x\mapsto\ClNrm(\beta')\beta^2 x$.
    \end{corollary}
    
    \begin{proof}       
    	We assume $\KK = \QQ$ to simplify notation. 
        Due to orthogonality, the subgroups $\sClGr_{\L}$ and $\sClGr_{\oL}$ commute and $\sClGr_{\L}\sClGr_{\oL}$ is a subgroup of $\sClGrV$. 
        Notice that $\sClGr_{\L}(\QQ)$ preserves $\L$ and $\nu$ (see \eqref{eq:ses-of-sClGr}) and acts trivially on $\oL$ as $vw=-wv$ for any $v \in \L$ and $w \in \oL$.         
        Hence, the intersection $\sClGr_{\L}(\QQ)\sClGr_{\oL}(\QQ)\cap\SpinV(\QQ)$ is a subgroup of $\HH_{\nu}(\QQ)$ and, by Zariski density, $(\sClGr_{\L}\sClGr_{\oL}) \cap\SpinV \subseteq \HH_{\nu}$.
        By dimension comparison, equality holds and, in particular,  $\HH_{\nu}$ is isogenous to $\Spin_{\L}\times \Spin_{\oL}$.
        Notice that \eqref{eq:char-stabilizer} claims a decomposition for $\QQ$-points (recall that $\KK = \QQ$ for simplicity).
        
        To get the decomposition, observe first that $\HH_{\nu}(\QQ)$ preserves $\L$ and $\oL$ by definition.
        Thus, for $\alpha\in\HH_{\nu}(\QQ)$ we have $\rest{\Adj_\alpha}{\L} \in \OO_{\Q|_{L}}(\QQ)$ and $\rest{\Adj_\alpha}{\oL} \in \OO_{\Q|_{\oL}}\!(\QQ)$.
    	In fact, since $\alpha$ preserves $\nu$, we get $\rest{\Adj_\alpha}{\L} \in \SO_{\Q|_L}(\QQ)$ and so also $\rest{\Adj_\alpha}{\oL} \in \SO_{\Q|_{\oL}}\!(\QQ)$.
        Using \eqref{eq:ses-of-sClGr} we can pick preimages $\beta\in\sClGr_{\L}(\QQ)$ and $\beta'\in\sClGr_{\oL}\!(\QQ)$ whose conjugation actions project to $\rest{\Adj_\alpha}{\L}$ and $\rest{\Adj_\alpha}{\oL} $, respectively.
        Then $\alpha$ and $\beta\beta'$ have the same image in $\SO_\Q$ and, thus, differ by a scalar multiple, say $a\in\QQ^\times$. 
        We have found a decomposition $\alpha = (a\beta)\beta'\in\sClGr_{\L}(\QQ)\sClGr_{\oL}\!(\QQ)$ where $\ClNrm(a\beta)\ClNrm(\beta') = \ClNrm(\alpha) = 1$. This establishes \eqref{eq:char-stabilizer}.
        
        For the last assertion, note first that the action of $\beta\beta'\in\HH_\nu(\QQ)$ on $x\in\L$ is given by
        \[ \beta\beta' x (\beta\beta')^{-1} = \beta x \frac{\stdinv(\beta)}{\ClNrm(\beta)} = \frac{\beta^2}{\ClNrm(\beta)}x = \ClNrm(\beta')\beta^2x, \]
        where the first equality follows since $x$ commutes with $\beta'\in\sClGr_{\oL}(\QQ)$, and since $\beta^{-1}=\stdinv(\beta)\ClNrm(\beta)^{-1}$, the second from $\beta x = x\stdinv(\beta)$ and the third from $\ClNrm(\beta)\ClNrm(\beta')=1$.
        The \namecref{cor:stabilizer-subgroup} follows.
    \end{proof}

\subsubsection{Klein vectors for oriented planes}\label{sec:Klein for planes}

    Using the Klein map, we associate to every non-degenerate rational oriented plane $\L$ in the vector space $\V$ a Klein vector with \highlight{integral} coefficients.
    To do this, fix a $\ZZ$-lattice $\VZ$ in $\V$ of full rank with $\Q(\VZ)\subset\ZZ$.
    This gives rise to a lattice of integral points in the Clifford algebra,
    \begin{align}\label{eq:Cl0(V(Z))}
        \Cle[(\VZ,\Q)] \subseteq \Cl,
    \end{align}
    that is, the Clifford algebra associated to the quadratic $\ZZ$-module $(\VZ,\Q)$.
    (Notice that the construction of Clifford algebras over free $\ZZ$-modules of finite rank is completely analogous to the one given after \Cref{def:clifford-alg}.)
    Let $\L\subseteq\V$ be a non-degenerate (necessarily rational) oriented plane. We pick an oriented $\ZZ$-basis $(v_1,v_2)$ of the two-dimensional oriented $\ZZ$-lattice $\L(\ZZ) \defeq \L\cap\VZ$. 
    This defines the pure wedge $v_1\wedge v_2\in\varV[nd]$ which corresponds to $\L$. Then we define the \highlight{Klein vector} associated to $\L$ by
    \[ \Klein(\L) \defeq \Klein(v_1\wedge v_2). \]
    The definition is independent of the choice of (oriented) $\ZZ$-basis by \Cref{prop:Klein-map}. Clearly, we get that $\Klein(\L)\in\Cle[(\VZ,\Q)]$ is a (traceless) vector with integral coefficients. We define the stabilizer subgroup of $\L$ as $\HH_{\L}\defeq\HH_{v_1\wedge v_2}<\SpinV$. Moreover, we define the pointwise stabilizer subgroup of $\L$ as
    \begin{align}\label{eq:defHLpt}
    \HH_{\L}^{pt} \defeq \{\alpha\in\HH_{\L}\st \alpha v \alpha^{-1} = v \text{ for all } v\in\L\}.
    \end{align}
    This is a $\QQ$-torus of absolute rank $1$.

    Given an oriented plane $\L$, we equip the orthogonal complement $\oL$ with an orientation as follows: if $(v_1,v_2)$ is an oriented basis of $\L$ and $(v_3,v_4)$ is an oriented basis of $\oL$, then $(v_1,\ldots,v_4)$ is an oriented basis of $\V$.

    There is a relation between the discriminant of a rational plane $\L\subseteq\V$ and the length of the associated Klein vector.

    \begin{lemma}\label{lem:length-of-Klein-vector}
		For any non-degenerate oriented plane $\L\subseteq\V$ we have 
        \begin{align*}
            \ClNrm(\Klein(\L)) = 4\disc(\L).
        \end{align*}
	\end{lemma}
	\begin{proof}
		This is a direct consequence of \eqref{eq:norm-Klein-vector}.
	\end{proof}

    \begin{remark}\label{rem:special-ClGr}
        For later use we remark that as $\Klein(\L)$ generates the two-dimensional $\QQ$-algebra $\Cle[(\L,\QL)]$, we get
        \begin{align}\label{eq:special-ClGr}
            \sClGr_{\L}(\KK) = \Cle[(\L,\QL)]^\times = \{a + b\Klein(\L)\st a^2 + 4\disc(\L)b^2\neq 0,~ a,b\in\KK\}.
        \end{align}
        Thus, $\sClGr_{\L}$ is a $\QQ$-split torus if and only if $-\disc(\L)$ is a square. In particular, this gives a concrete description of $\HH_\L(\KK)$ via \eqref{eq:char-stabilizer}. 
    \end{remark}
    
\subsection{Comparison of constructions of Klein vectors} \label{subsec:equivalences-of-Klein-vector-constructions}
    The construction of the Klein map using Clifford algebras in \Cref{prop:Klein-map} is a natural generalization of the construction of the Klein vectors in \cite{aka_planes_2021}.

    We quickly recall the setting of \cite{aka_planes_2021}. The underlying quadratic form $\Q_0$ on $\QQ^4$ is the sum of squares $x_1^2 + x_2^2 + x_3^2 + x_4^2$. This form is realized as the norm form on the $\QQ$-algebra $B_0 = \left(\frac{-1,-1}{\QQ}\right)$ of Hamiltonian quaternions. Denote by $\conj\colon B_0\to B_0$ the canonical conjugation of the quaternion algebra $B_0$. For a non-degenerate rational plane $\L\subseteq B_0$ two associated Klein vectors $a_1(\L)$ and $a_2(\L)$ are defined as
    \begin{align*}
    	a_1(\L) \defeq v_1\conj[v_2] - \tfrac{1}{2}\Tr(v_1\conj[v_2]) =\tfrac{1}{2}(v_1\conj[v_2] - v_2\conj[v_1]),\\
    	a_2(\L) \defeq \conj[v_2]v_1 - \tfrac{1}{2}\Tr(\conj[v_2]v_1) = \tfrac{1}{2}(\conj[v_2]v_1 - \conj[v_1]v_2),
    \end{align*}
    where $\{v_1,v_2\}$ is a basis of $\L(\ZZ)$.
    
    Observe that the embedding $\iota:B_0\to\Mat[2](B_0)$ (of $\QQ$-vector spaces and not of $\QQ$-algebras) given by $\iota(v) = \smat{}{v}{\conj[v]}{}$ satisfies
    \begin{align*}
        \iota(v)^2 = \begin{pmatrix}&v\\\conj[v]&\end{pmatrix}\begin{pmatrix}&v\\\conj[v]&\end{pmatrix} = \Q(v) \begin{pmatrix}1&\\&1\end{pmatrix}.
    \end{align*}
    Thus, by the universal property \eqref{eq:universal-property} of the Clifford algebra there is a homomorphism $\iota'\colon\Cliff(B_0,\Q_0)\to\Mat[2](B_0)$ of $\QQ$-algebras satisfying $\iota'(v) = \iota(v)$.
    By dimension comparison, $\iota'$ is an isomorphism.
    Moreover, we get
    \begin{align}
        \iota'(\Klein(\L)) &= \iota'([v_1,v_2])= \begin{pmatrix}&v_1\\\conj[v_1]&\end{pmatrix}\begin{pmatrix}&v_2\\\conj[v_2]&\end{pmatrix} - \begin{pmatrix}&v_2\\\conj[v_2]&\end{pmatrix}\begin{pmatrix}&v_1\\\conj[v_1]&\end{pmatrix}\nonumber\\
        &=  \begin{pmatrix}v_1\conj[v_2] - v_2\conj[v_1]&\\&\conj[v_1]v_2 -\conj[v_2]v_1&\end{pmatrix} = \begin{pmatrix}2a_1(\L)&\\& -2a_2(\L)\end{pmatrix}.\label{eq:iota(Klein)}
    \end{align}
    This shows that the Klein map in \Cref{prop:Klein-map} indeed generalizes the construction of the Klein vectors in \cite{aka_planes_2021}.
      
    In fact, one can extend the above setup to quadratic forms whose discriminant is \highlight{not a square} in $\QQ^\times$ by using quaternion algebras over \highlight{quadratic fields} as follows. Let $B$ be a quaternion algebra over $\QQ$. 
    We write $\B$ for the affine variety $B$ defined by $\B(A) = B\otimes A$ for any $\QQ$-algebra $A$. Let $F/\QQ$ be a quadratic extension and write $\tau \in \mathrm{Gal}(F/\QQ)$ for the non-trivial element. The Galois group acts on $\mathbf{B}(F)$ in the obvious way; we denote it by $x \mapsto \GalInv{x}$.
    Set
    \begin{align*}
        \V = \{x \in \mathbf{B}(F): \GalInv{x} = -\bar{x} \}.
    \end{align*}
    As we see below, this is a four-dimensional vector space over $\QQ$. Let $\Q$ be the restriction of the norm form on $\B(F)$ to $\V$. Note that for any $v \in\V$
    \begin{align*}
        \GalInv{\Q(v)} = \GalInv{v} \GalInv{\conj[v]} = \conj[v] v = v \conj[v] = \Q(v)
    \end{align*}
    and so $\Q$ takes values in $\QQ$. Explicitly, if for $a,b \in \QQ^\times$
    \begin{align*}
        B = \left(\frac{a,b}{\QQ}\right)
    \end{align*}
    with generators $j_1$, $j_2$, and $j_3 = j_1j_2$ and if $F = \QQ(\sqrt{d})$, then
    \begin{align*}
        \V = \QQ\sqrt{d} 1 + \QQ j_1 + \QQ j_2 + \QQ j_3
    \end{align*}
    and the quadratic form is given by
    \begin{align}\label{eq:explicit-qqf}
        \Q\big(x_0\sqrt{d} + x_1 j_1 + x_2 j_2 + x_3 j_3\big) = d x_0^2 - a x_1^2 -b x_2^2 + ab x_3^2.
    \end{align}
    In particular, the discriminant of $\Q$ is not a square and is equal to the discriminant of $F$ up to squares. In fact, any integral quaternary quadratic form of non-square discriminant is rationally equivalent to a quadratic form as above (but not integrally, in general).
    
    We again consider the embedding $\iota:\V\to\Mat[2](\B(F))$ given by $v\mapsto \smat{}{v}{\conj[v]}{}$ and note that the image is contained in the subalgebra
    \begin{align}\label{eq:emb-of-V}
        \left\{\begin{pmatrix}x_0&x_1\\-\GalInv{x_1}&\GalInv{x_0}\end{pmatrix}\st x_0,x_1\in\B(F)\right\} \subseteq \Mat[2](\B(F)),
    \end{align}
    since $\conj[v] = -\GalInv{v}$ by definition of $\V$. The same calculation as above gives $\iota(v)^2 = \Q(v)\smat{1}{}{}{1}$ and, thus, by the universal property \eqref{eq:universal-property} of the Clifford algebra there is an isomorphism of $\QQ$-algebras
    \begin{align}\label{eq:galois-twist-clalg}
        \iota'\colon\Cl\to\left\{\begin{pmatrix}x_0&x_1\\-\GalInv{x_1}&\GalInv{x_0}\end{pmatrix}\st x_0,x_1\in\B(F)\right\}
    \end{align}  
    with $\iota'(v) = \iota(v)$. A similar identity as in \eqref{eq:iota(Klein)} also holds.
    By restriction of $\iota'$, we have an isomorphism of algebras
    \begin{align}\label{eq:galois-twist-clealg}
        \iota'\colon\Cle \to \left\{\begin{pmatrix}x_0&\\&\GalInv{x_0}\end{pmatrix}\st x_0\in\B(F)\right\} \simeq \B(F)
    \end{align}
    and an isomorphism of vector spaces
    \begin{align*}
        \iota'\colon\Clo \to \left\{\begin{pmatrix}&x_1\\-\GalInv{x_1}&\end{pmatrix}\st x_1\in\B(F)\right\} \simeq \B(F).
    \end{align*}
    Note that for any $x\in \B(F)^\times$ and $v\in\V$
    \begin{align*}
    \begin{pmatrix}
    x & 0 \\ 0 & \GalInv{x}
    \end{pmatrix}
    \begin{pmatrix}
    0 & v \\ -\GalInv{v} & 0
    \end{pmatrix}
    \begin{pmatrix}
    x & 0 \\ 0 & \GalInv{x}
    \end{pmatrix}^{-1}
    =\begin{pmatrix}
    0 & xv\GalInv{x}^{-1} \\ -\GalInv{(xv\GalInv{x}^{-1})} & 0
    \end{pmatrix}.
    \end{align*}
    Therefore, under the identifications in \eqref{eq:galois-twist-clalg} and \eqref{eq:galois-twist-clealg}, the action of $g\in \B(F)^\times$ on $\V$ is given by 
    \begin{align}\label{eq:twistedGaloisaction}
    g.v = gv(\GalInv{g}^{-1})
    \end{align}
    for all $v\in\V$.

    Via \eqref{eq:galois-twist-clealg}, the involution $\sigma$ induces the standard involution on $\B(F)$ (see e.g. Proposition~\ref{prop:structure-of-Cle}).
    Therefore, $\ClNrm$ corresponds to the norm form on $\B(F)$ and $\SpinV(\QQ)$ is identified with $\B^1(F)$.

\subsection{Specialization to our standard signature \texorpdfstring{$(1,3)$}{(1,3)}-form}\label{sec:specialization}
    We consider now $\V = \QQ^4$ equipped with the quadratic form
    \[ \Q(x_1,x_2,x_3,x_4) = -x_1^2 - x_2^2 - x_3^2 + x_4^2, \]
    and the integral structure $\VZ = \ZZ^4 = \langle e_1,e_2,e_3,e_4\rangle_\ZZ$.
    Define
    \begin{align*}
        \SpinV(\ZZ) &= \big\{\alpha\in\SpinV(\QQ)\st \Adj_{\alpha}(\VZ) = \VZ\big\}
    \end{align*}
    and similarly for $\SpinV(\ZZ_p)$ for every prime $p$.
    In view of \eqref{eq:explicit-qqf}, we take $F=\QQ(i)$ and $\B=\Mat[2]=\big(\frac{1,1}{\QQ}\big)$ in the previous section.
    Via the linear map given by
    \begin{align}\label{eq:basis over Mat}
        e_1 \mapsto\mat{i}{}{}{i},\ e_2 \mapsto \mat{1}{}{}{-1},\ e_3 \mapsto\mat{}{1}{1}{},\ e_4 \mapsto\mat{}{1}{-1}{},
    \end{align}
    the quadratic space $(\V,\Q)$ is identified with the $\QQ$-subspace $\{\GalInv{x} = -\bar{x}\}$ of the quaternion algebra $\Mat[2](\QQi)$ equipped with the quadratic form $\det$.
    Under this identification we have
    \begin{align}\label{eq:integral-points}
        \VZ &=
            \Big\{ \begin{pmatrix}
            x_1i + x_2 & x_3 + x_4 \\ x_3-x_4 & x_1i - x_2
            \end{pmatrix}: x_1,x_2,x_3,x_4 \in \ZZ \Big\}.
    \end{align}
    Relating to the introduction, note that $\V$ is identified with the space $\mathcal{H}$ of Hermitian matrices in equivariant manner by right-multiplication with $\smat{0}{1}{-1}{0}$.
    It is more convenient for us to work with $V$ instead of $\mathcal{H}$.
    
    We consider $(e_1,e_2,e_3,e_4)$ as an oriented basis of $\V$. We obtain that $\Cle\simeq\Mat[2](\QQi)$ and $\SpinV(\QQ) \simeq \SL_2(\QQi)$ by \eqref{eq:galois-twist-clealg}. 
    In fact, this gives an isomorphism of $\QQ$-groups
	\begin{align}\label{eq:spin-iso-sl}
        \SpinV &\to \Res{\QQi}(\SL_2).
	\end{align}
    (Strictly speaking, one ought to extend the discussion of \S\ref{subsec:equivalences-of-Klein-vector-constructions} from $\QQ$-points to $A$-points for any $\QQ$-algebra $A$.)
    The following \namecref{prop:iso-cliff} adds integrality information into this picture.
   
    \begin{proposition}\label{prop:iso-cliff}
	The image of $\SpinV(\ZZ)$ under \eqref{eq:spin-iso-sl} is conjugate to $\SL_2(\ZZi)$.
    Explicitly, it is given by
    \begin{align}\label{eq:ZZ-points}
    \mat{1+i}{1}{0}{1} \SL_2(\ZZi) \mat{1+i}{1}{0}{1}^{-1}.
        \end{align}
	\end{proposition}
    
    \begin{proof}
        By \eqref{eq:twistedGaloisaction} the action of $\SpinV(\QQ)$ on $V$ corresponds to the action of $\SL_2(\QQi)$ on $V$ through $g.v = gv(\GalInv{g}^{-1})$.
        Let $\order \subset \Mat[2](\QQi)$ be the subalgebra corresponding to $\Cle[(\V(\ZZ),\Q)]$ via \eqref{eq:galois-twist-clealg}; see
        \eqref{eq:Cl0(V(Z))} for the definition of $\Cle[(\V(\ZZ),\Q)]$.
        Equivalently, $\order$ is the $\ZZ$-algebra generated by $v\GalInv{w}$ for $v,w\in\VZ$ which can be computed to be
        \begin{align}\label{eq:basis-order}
            \order &=
            \Big\{ \begin{pmatrix}
            x_1 & x_2 \\ x_3 & x_4
            \end{pmatrix}: x_1,x_2,x_3,x_4 \in \ZZ[i],\ x_1+x_4,x_2+x_3 \in 2\ZZ[i]\Big\}.
        \end{align}
        If $g \in \SL_2(\QQi)$ preserves $\VZ$ (as does the image of an element in $\SpinV(\ZZ)$ under \eqref{eq:spin-iso-sl}), then it preserves $\order$ under conjugation.
        Using also that $\det(g)=1$, $g$ belongs to the group $\operatorname{N}^1(\order)$ of  norm-one elements of the normalizer of the order~$\order$.
        
        We begin by computing $\operatorname{N}^1(\order)$.
        So let
        \begin{align*}
            g = \begin{pmatrix}a&b\\c&d\end{pmatrix}\in \SL_2(\QQi).
        \end{align*}
        By explicit computation using the definition of the normalizer, we obtain that $g \in \operatorname{N}^1(\order)$ if and only if
        \begin{align}\label{eq:coefficientsonenormalizer}
            2ad,2ab,2cd,2bc, a^2-b^2, \tfrac{1}{2}(a^2-b^2+c^2-d^2),ac-bd, 2bd,2b^2,2d^2 \in \ZZ[i].
        \end{align}
        We record a few elementary properties of the conditions in \eqref{eq:coefficientsonenormalizer}. They imply $a^2,b^2,c^2,d^2 \in \frac{1}{2}\ZZ[i]$ and hence $a,b,c,d \in \frac{1}{1+i}\ZZ[i]$. Moreover, $(a+b)^2 = a^2 + 2ab + b^2 = (a^2-b^2) + 2ab +2b^2 \in \ZZ[i]$ so that $a+b \in \ZZ[i]$; similarly, $c+d \in \ZZ[i]$.
        Furthermore, we have $a+b-c-d \in (1+i)\ZZ[i]$, since
        \begin{align*}
            (a+b-c-d)(a-b-c+d) &= (a-c)^2-(b-d)^2\\
            &= a^2-b^2+c^2-d^2 - 2(ac-bd) \in 2\ZZ[i]
        \end{align*}
        and $(a+b-c-d) + (a-b-c+d) = 2a-2c \in (1+i)\ZZ[i]$. Lastly, since  $\operatorname{N}^1(\order)$ is a group, $g^{-1} = \smat{d}{-b}{-c}{a}$ satisfies the analogous relations and in particular $b-d \in \ZZ[i]$.
        Overall, it follows that
        \begin{align*}
        \operatorname{N}^1(\order) = \gamma^{-1} \SL_2(\ZZ[i])\gamma \quad \text{where} \quad
            \gamma^{-1} \defeq \mat{1+i}{1}{0}{1} .
        \end{align*}

        A priori, one might suspect the unit normalizer group $\operatorname{N}^1(\order)$ to be slightly larger than $\SpinV(\ZZ)$; this turns out not to be the case, that is we claim that $\operatorname{N}^1(\order) = \SpinV(\ZZ)$. 
        To verify this, notice that it is enough to find a single element $v_0 \in \V(\ZZ)$ of invertible determinant with $gv_0(\GalInv{g}^{-1}) \in \V(\ZZ)$ for all $g \in \operatorname{N}^1(\order)$. Indeed, in view of \eqref{eq:galois-twist-clalg} the lattice $\order v_0$ in $\Mat[2](\QQi)$ is then invariant under $v \mapsto gv\GalInv{g}^{-1}$ and hence so is its intersection with $\V$, which is $\V(\ZZ)$.
        Alternatively, in the Clifford algebra viewpoint the above argument says that whenever $g$ preserves $\Cle[(\VZ,\Q)]$ and there is a vector $v_0 \in \V(\ZZ)$ of unit quadratic value with $\Adj_g(v_0) \in \V(\ZZ)$ then $\Cle[(\VZ,\Q)] v_0 = \Clo[(\VZ,\Q)]$ is also invariant and hence so is $\V(\ZZ)$.
        Take
        \begin{align*}
        v_0 = \begin{pmatrix}i&\\&i\end{pmatrix} \in \V(\ZZ).
        \end{align*}
        Note that $\V \cap \order = \V(\ZZ)$ by \eqref{eq:integral-points} and \eqref{eq:basis-order}.
        Using that $\V$ is invariant, we have $gv_0\GalInv{g}^{-1} \in \V(\ZZ)$ if and only if $gv_0\GalInv{g}^{-1} \in \order$. The latter is equivalent to $g\GalInv{g}^{-1} \in \order$ as $v_0$ is central.
        Now let
        \[ g = \gamma^{-1}\begin{pmatrix}a&b\\c&d\end{pmatrix}\gamma \in \operatorname{N}^1(\order) \]
        where $a,b,c,d\in\ZZi$.
        Then a simple calculation gives
        \begin{align*}
        g\GalInv{g}^{-1} = \gamma^{-1}\begin{pmatrix}
        ia\GalInv{d} - b\GalInv{c}&b\GalInv{a} - ia\GalInv{b}\\
        ic\GalInv{d} - d\GalInv{c}&d\GalInv{a} - ic\GalInv{b}\end{pmatrix}\GalInv{\gamma}.
        \end{align*}
        If we write $a\GalInv{b} = x_1+iy_1$, $c\GalInv{d} = x_2+iy_2$, and $ia\GalInv{d} - b\GalInv{c} = x_3 + iy_3$ for integers $x_1,x_2,x_3, y_1, y_2, y_3$, then
        \begin{align*}
        g\GalInv{g}^{-1} = \begin{pmatrix}
        i(x_1+y_1) -ix_3+y_3 & -i(x_1+y_1)\\
        i(x_1+y_1) + 2i (x_2+y_2) & -i(x_1+y_1) -ix_3+y_3
        \end{pmatrix}
        \end{align*}
        which is clearly an element of $\order$. This proves the proposition.
    \end{proof}
    
	By abuse of notation, we will henceforward view points of the spin group $\SpinV$ as points of $\Res{\QQi}(\SL_2)$ using the isomorphism \eqref{eq:spin-iso-sl} above.

    \begin{example}\label{exp:stab L0}
    Let $\L_0$ be the oriented $\QQ$-subspace of $\V$ spanned by $e_1,e_2$. Under~\eqref{eq:basis over Mat}, we view $\L_0$ equivalently as
    \begin{align}\label{eq:stab-Lo}
    \L_0 = \Big\{\mat{*}{0}{0}{*}\Big\} \subset \V = \{x \in \Mat[2](\QQi): \GalInv{x} = -\bar{x}\}.
    \end{align}
    By~\eqref{eq:galois-twist-clealg} and \eqref{eq:twistedGaloisaction}, the stabilizer group $\HH_{\L_0}<\SpinV$ is then identified with 
    \begin{align*}
    \Big\{g \in \mathrm{Res}_{\QQi/\QQ}(\SL_2): g\mat{*}{}{}{*}\GalInv{g}^{-1} \subset \mat{*}{}{}{*},\ g \text{ preserves orientation}\Big\}
    \end{align*}
    and is therefore equal to the full diagonal subgroup.
    In particular, $\HH_{\Lo}(\RR) = A_\CC$.
    From \eqref{eq:stab-Lo}, it is also clear that $\HH_{\L_0}^{pt}(\RR) = A$ and $\HH_{\L_0^\perp}^{pt}(\RR) = A_{\CC}^1$.
    The Klein vector $\Klein(L_0)$ spans the line through $\mathrm{diag}(1,-1)$ in traceless $(2\times2)$-matrices over $\QQi$; its stabilizer in $\mathrm{Res}_{\QQi/\QQ}(\SL_2)$ under conjugation is also the diagonal subgroup as is, of course, implied by Proposition~\ref{prop:Klein-map}. 
    \end{example}
    
    Recall that the Klein map defined in \Cref{prop:Klein-map} takes values in the three-dimensional $\QQi$-vector space of traceless elements of $\Cle$ or, equivalently, in traceless matrices in $\Mat[2](\QQi)$. We denote this space by
	\begin{align}\label{eq:W-matrices}
	    \W \defeq \{\alpha\in\Cle\st \stdinv(\alpha) = -\alpha\} \simeq \Mat[2]^0(\QQi).
	\end{align}
	The integer points in $\W$ are defined via the integral structure on $\V$, that is
	\[ \WZ \defeq \W \cap \Cle[(\VZ,\Q)] \simeq \left\{\begin{pmatrix}x_1&x_2\\x_3&-x_1\end{pmatrix}\st ~ x_1,x_2,x_3\in\ZZi,\ x_2+x_3\in2\ZZi\right\} \]
	which is a three-dimensional $\ZZ[i]$-lattice. When we restrict the norm form (i.e.~the determinant) $\ClNrm$ to $\W$ we obtain a quadratic form $\QF\colon\W\to\QQi$. The conjugation action yields an exact sequence
	\begin{align}\label{eq:spin-to-SOQF}
	    1\to\{\pm1\} \to\SpinV \to \Res{\QQ(i)}(\SO_{\QF})\to 1.
	\end{align}
    Moreover, it is clear that $\SpinV(\ZZ)$ maps into $\SO_{\QF}(\ZZi)$ since $\SpinV(\ZZ)$ preserves $\VZ$ and thus also $\WZ$.
	
	\begin{remark}
		The fact that $\WZ$ is a \highlight{free} three-dimensional $\ZZ[i]$-lattice is one of the reasons why we restrict to the quadratic form $\Q(x_1,x_2,x_3,x_4) = -x_1^2 - x_2^2 - x_3^2 + x_4^2$. In general, the integer points of the even Clifford algebra are not a free module over the integers points of its center. For example, for the diagonal quadratic form $ 2x_1^2 + x_2^2 + x_3^2 + 5x_4^2$ the integer points of the associated even Clifford algebra are \highlight{not} a free module over the integers points of its center, which is $\ZZ[\sqrt{10}]$. 
	\end{remark}

    A pleasant property of the quadratic form $\Q(x_1,x_2,x_3,x_4) = -x_1^2 - x_2^2 - x_3^2 + x_4^2$ is that its discriminant $-1$.
    By the arguments in \cite[Sec.~3.2]{aka_equidistribution_2021}, we have for any oriented plane $\L$
    \begin{align*}
        \disc(\L) = -\disc(\oL).
    \end{align*}
    We recall that $\Klein(\oL) = z\Klein(\L)$ holds for some traceless $z\in\Z$, see \eqref{eq:orth-klein}. The definition of the (orientation of the) orthogonal complement $\oL$ then yields the following:

    \begin{lemma}\label{lem:KL-to-KoL}
        We have
        \begin{align}\label{eq:Klein-vector-of-orthogonal}
            \Klein(\oL) = -\sgn(\L)\ori\Klein(\L),
        \end{align}
        where $\ori\defeq e_1e_2e_3e_4\in\Z$ and $\sgn(\L)$ denotes the sign of $\disc(\L)$.
    \end{lemma}
    \begin{proof}
        Let $(v_1,v_2)$ and $(v_3,v_4)$ be oriented bases of $\L$ and $\oL$, respectively. Then, by definition of the orientation of $\V$ and $\oL$ we know that $(v_1,v_2,v_3,v_4)$ has the same orientation as $(e_1,e_2,e_3,e_4)$. This implies that the traceless element 
        \[  [v_1,v_2][v_3,v_4] = \Klein(\L)\Klein(\oL) = \Klein(\L)^2z = -4\disc(\L)z \]
        is a positive multiple of $\ori$, say $r\ori$ for $r\in\QQ_{>0}$. 
        Note that $r^2 = 16\abs{\disc(\L)\disc(\oL)}$ and so $r = 4\abs{\disc(L)}$.
        Thus, by \eqref{eq:KL-to-KoL} we get that
        \[ \Klein(\oL) = -\frac{[v_1,v_2][v_3,v_4]}{4\disc(\L)}\Klein(\L) = -\frac{r}{4\disc(\L)}\ori\Klein(\L) . \]
        Since $r$ is positive, we get $\frac{r}{4\disc(\L)} = \sgn(\L)$ proving the lemma.
    \end{proof}
    
    \begin{remark}
        A traceless element $z\in\Z(\Cle)^\times$ defines an orientation on the quadratic space $\V$ in a natural way. These traceless elements act as $-\Id$ on the vector space $\V$ and this can be considered as their unique defining property.
    \end{remark}
    
\subsection{Integrality properties of the Klein vectors}
	We will now analyze primitivity properties of the Klein vector $\Klein(\L)$ associated to an oriented rational plane $\L\subseteq\V$. 
    Choosing a basis $v_1,v_2$ of $\L$ consisting of integral vectors, it is immediate that $\Klein(\L)\in\Cl[(\VZ,\Q)]$. 
    However, even if $v_1,v_2$ are primitive vectors in $\VZ$ the Klein vector $\Klein(\L)$ might not be primitive as an element of the $\ZZi$-module $\WZ$.
	
	We fix the $\ZZ$-basis $\{e_1,e_2,e_3,e_4\}$ of $\VZ$ as well as the $\ZZi$-basis 
    \begin{align}\label{eq:fi basis}
    f_1 = e_1e_2,f_2=e_2e_3,f_3=e_1e_3
    \end{align}
    of the rank-$3$ free $\ZZi$-module $\WZ$. By convention, we write vectors in $\V$ and $\W$ as row vectors in $\QQ^4$ and $\QQi^3$, respectively. In particular, matrices act from the right.
	
	\begin{definition}\label{def:primitivity}
		We call an element $v\in\VZ$ \highlight{primitive} if $\QQ v\cap\VZ= \ZZ v$.
		Similarly, we call an element $w\in\WZ$ \highlight{primitive} if $\QQi w\cap\WZ = \ZZi w$.
	\end{definition}
	
	Since $\ZZ[i]$ is a principal ideal domain, an element $w\in\WZ$ is primitive if and only if it is an element of a basis of $\WZ$. 
    For every $w\in\WZ$ there is a primitive element $\tilde{w}$ in the line $\QQi w$ which is unique up to multiplication by a unit $\ZZi^\times$.
	
	\begin{definition}[Primitive Klein vector]
		Let $\L\subseteq\V$ be a non-degenerate plane. We denote by $\prKlein(\L)$ a primitive vector in $\QQi\Klein(\L)$.
        
        We choose the primitive elements of rational, non-degenerate planes $\L$ in a compatible way with respect to \eqref{eq:Klein-vector-of-orthogonal}, 
        that is, if $\oL$ denotes the orthogonal complement of $\L$ and \reflectbox{$\L$} is the same subspace as $\L$ but with reversed orientation, then we require that
        \[ \prKlein(\oL) = -\sgn(\L)\ori\prKlein(\L),\quad \prKlein(\reflectbox{$\L$}) = - \prKlein(\L),\quad \prKlein(\orth{\reflectbox{$\L$}}) =\sgn(\L)\ori\prKlein(\L). \]
	\end{definition}

    Assume $(v_1,v_2)$ is an oriented $\ZZ$-basis of $\LZ$. Then
    \[ \Klein(\L) = [v_1,v_2] = v_1v_2 - v_2v_1 = 2v_1v_2 - \bQ(v_1,v_2), \]
    and, since $\bQ$ takes values in $2\ZZ$ when restricted to $\VZ$, we notice that $\Klein(\L)\in 2\WZ$. 

    \begin{lemma}\label{lem:prim-KL}
        Suppose that $D=\disc(L) \geq 1$ is square-free. Then $\Klein(\L)/2 \in r \WZ$ implies $r \mid (1+i)$.
        Moreover, if $D$ is odd, we have that $\Klein(\L)/2$ is primitive.
    \end{lemma}

    Thus, if $D$ is odd and square-free, we define $\prKlein(\L) = \Klein(\L)/2$.
    Note that $D$ odd and square-free is equivalent to $D$ being square-free over the PID $\ZZi$.
    
    \begin{proof}
        Since $\ZZi$ is a PID, we may assume that $r$ is a prime power.
        Taking the determinant, the assumption and Lemma~\ref{lem:length-of-Klein-vector} imply that $r^2 \mid D$.
        If $r$ is a power of a rational prime, we conclude; assume otherwise.
        If $r\nmid 2$, $r$ and its conjugate are coprime ($2$ is the only ramified prime in $\QQi$) and we obtain that the norm of $r$ squared divides $D$; this is a contradiction.
        If $r \mid 2$ we have either $r\mid (1+i)$ or $2 \mid r$. In the later case, we again obtain a contradiction.
    \end{proof}
    

\subsection{The orthogonal complement of the Klein vector}\label{sec:orth lattice Klein}

    Observe that the quadratic form $\QF$ is equal to the sum of three squares in the $\QQi$-basis $f_1=e_1e_2,f_2=e_2e_3,f_3=e_1e_3$. 
    Note that $f_1,f_2,f_3$ is an $\ZZi$-basis of $W(\ZZi)$.
    Let $\bQF$ be the bilinear form associated to $\QF$ similarly defined as \eqref{eq:bilinear-form}. For every non-zero $w\in\W$ we let $\orth{w} \defeq \{w'\in\W\st \bQF(w,w') = 0\}$.

    Our geometric setup dictates a notion of orientation on $\W$ which we now introduce (the reader should beware that general $\QQi$-spaces do not come with a notion of orientation).
    Considering that we fixed the ordered basis $(f_1,f_2,f_3)$ of $\W$, we say that an ordered basis $(w_0,w_1,w_2)$ of $\W$ is \highlight{oriented} if the determinant of the matrix
    \[ \begin{pmatrix}\text{---}&w_0&\text{---}\\\text{---}&w_1&\text{---}\\\text{---}&w_2&\text{---}\end{pmatrix} \in\GL_3(\QQi) \]
    is in $\QQ_{>0}$. 
    In particular, if $(w_0,w_1,w_2)$ is an oriented basis of $\WZ$ then the above matrix is in $\SL_3(\ZZi)$.
    
    Any primitive vector $w\in\WZ$ can be complemented to an oriented basis of $\WZ$. Further, by primitivity there exists $u_w\in\WZ$ such that $\tfrac{\bQF(w,u_w)}{2}=1$ and we note that $u_w$ is unique up to adding elements of $\orth{w}\cap\WZ$. We call an ordered basis $(w_1,w_2)$ of $\orth{w}$ \highlight{oriented}, if $(u_w,w_1,w_2)$ is an oriented basis of $\W$. Clearly, this definition does not depend on the choice of $u_w$. Further, observe that two oriented bases of $\orth{w}$ differ by an element in $\GL_2(\QQi)$ with determinant in~$\QQ_{>0}$.

    For any primitive $w\in\WZ$ we define the \highlight{oriented orthogonal lattice} $\Lambda_w $ in $\orth{w}$ by
    \begin{align}\label{eq:ori-orth-lattice}
        \Lambda_w \defeq \orth{w}\cap\WZ,
    \end{align} 
    together with an oriented $\ZZi$-basis $(w_1,w_2)$ of $\Lambda_w$. 
    Note that such a basis exists: for any $\ZZi$-basis $(w_1,w_2)$ of $\Lambda_w$ there exists $\alpha \in \ZZi^\times$ so that $(w_1,\alpha w_2)$ is oriented.
    Similarly as before, we define the discriminant of the oriented lattice $\Lambda_w$ in $\W$ as
    \[ \disc(\Lambda_w) \defeq \det\begin{pmatrix}\QF(w_1)&\bQF(w_1,w_2)/2\\\bQF(w_2,w_1)/2&\QF(w_2)\end{pmatrix}, \]
    and we note that this definition is independent of the chosen oriented basis.
    
     \begin{lemma}\label{lem:disc-of-orth}
        Let $w\in\WZ$ be primitive. Then $w = \QF(w)u_w + bw_1 + cw_2$ for some $b,c\in\ZZi$ where $(w_1,w_2)$ is an oriented basis of $\Lambda_w$. Moreover, the oriented lattice $\Lambda_w$ satisfies
    	\[  \disc(\Lambda_w) = \QF(w). \]
    \end{lemma}
    \begin{proof} 
        Since $(u_w,w_1,w_2)$ is a basis of $\W$ so we may write $w = au_w + bw_1 + cw_2$.
        Then 
        \[ \QF(w) = \frac{\bQF(w,w)}{2} = \frac{\bQF(w,au_w + bw_1 + cw_2)}{2} = a\frac{\bQF(w,u_w)}{2} = a \]
        and so we get $w = \QF(w)u_w + bw_1 + cw_2$ for some $b,c\in\ZZi$.
        
        Let $h \in \Mat[3](\ZZi)$ be the matrix with rows $w,w_1,w_2$.
        Then $\det(h)=\QF(w)$ by the previous assertion. On the other hand,
        \[ \det(h)^2 = \det(hh^t) = \det\begin{pmatrix}\QF(w)& 0&0\\0&\QF(w_1)&\bQF(w_1,w_2)/2\\0&\bQF(w_1,w_2)/2&\QF(w_2)\end{pmatrix} \]
        is equal to $\QF(w)\disc(\Lambda_w)$. It follows that $\disc(\Lambda_w) = \QF(w)$.
    \end{proof}
    
    For completeness, we also record the following lemma that does not rely on the above integrality discussions.
    
    \begin{lemma}\label{lem:compl of Klein}
    	The map 
        \begin{align*}
        \Cl[(\L,\Q|_L)]\otimes\Cl[(\oL,\Q|_{\oL})]\to\Cl,\ v_1\otimes v_2 \mapsto v_1v_2
        \end{align*}
        defines a grading-preserving isomorphism and restricts to an isomorphism $\L\otimes\oL\to\oKL$ for any plane $L \subset \V$.
    \end{lemma}
    \begin{proof}
        One may readily verify that $\phi$ is indeed a grading-preserving isomorphism.
        For any $v_1 \in \L$ and $v_2 \in \oL$, we have $\stdinv(v_1v_2) = v_2v_1 = -v_1v_2$ due to orthogonality of $\L$ and $\oL$, so the image $\phi(\L\otimes\oL)$ consists of traceless elements in $\Cle$, that is, $\phi(\L\otimes\oL)\subseteq\W$. Observe that $\bQ[\QF](\Klein(\L),w) = 0$ is equivalent to $\Klein(\L)w = - w\Klein(\L)$, therefore
        \begin{align*}
            \oKL = \{\alpha\in W\st \Klein(\L)\alpha = - \alpha\Klein(\L) \}.
        \end{align*}
        It follows that $\phi(\L\otimes\oL)$ is a subset of $\oKL$ since $\L$ and $\oL$ are subspaces of the $(-1)$- and $(+1)$-eigenspaces of $\Adj_{\Klein(L)}$, respectively. By a simple dimension count (over $\QQ$) we conclude that $\phi(\L\otimes\oL) = \oKL$ as claimed.
    \end{proof}


\section{Definition of shapes, CM points and periodic geodesics}\label{sec:shapes}

    We use the notation fixed in \S\ref{sec:Clifford-algebra-and-Klein-map}. 
    Throughout, $D$ is an odd square-free positive integer (see also~\Cref{rem:sqfree} below).
    Recall from \eqref{eq:Def-of-JD} that $\P_D$ denotes the collection of rational planes of discriminant $D$.
    It satisfies $\QF(\prKlein(\L))= D$ by \Cref{lem:disc-of-orth}. 
    We set
    \begin{align}\label{eq:Klein-orth-lattice}
        \Lambda_{\L} \defeq \Lambda_{\prKlein(\L)}
    \end{align}
    and call it the \highlight{Klein orthogonal lattice} associated with $\L$ (see \eqref{eq:ori-orth-lattice}).
    Let
    \[ \X = \Tan(\Y)  = \lquot{\SL_2(\ZZ)}{\SL_2(\RR)} \]
    be the unit tangent bundle of the modular curve $\Y = \SL_2(\ZZ)\backslash\mathbb{H}^2$ and let
    \[ \X_{\QQi} \defeq \lquot{\SL_2(\ZZi)}{\SL_2(\CC)} \]
    be the frame bundle of the Bianchi orbifold $\Y_{\QQi} =\SL_2(\ZZi)\backslash\HypSpace^3$ (cf.~\eqref{eq:Bianchiorbifold}).
    
    Let us summarize the following section. 
    To $\L\in\P_D$ we will associate an orbit
    \begin{align}\label{eq:ass-orbit}
        \GG(\ZZ)\TT_\L(\RR)y_{\L} \subseteq \X_{\QQi}\times\X\times\X\times\X_{\QQi}.
    \end{align}
    Here, $\GG$ is as in \eqref{eq:defG} and $\TT_\L< \GG$ will be a $\QQ$-torus of absolute rank two of the~form
    \begin{align*}
        \TT_\L = \{(h,\theta_{\L}(h),\theta_{\oL}(h),\theta_{\oKLZ}(h))\st h\in\HH_\L\},
    \end{align*}
    where $\HH_\L<\SpinV\simeq \mathrm{Res}_{\QQi/\QQ}(\SL_2)$ is the stabilizer subgroup of $\L$, and $\theta_{\L},\theta_{\oL}$, $\theta_{\oKLZ}$ are homomorphisms from $\HH_\L$ into the corresponding factor groups of $\GG$. 
    Projections from the right-hand side of \eqref{eq:ass-orbit} onto each of the irreducible quotients will yield respectively,
    \begin{enumerate}
        \item a periodic geodesic in the Bianchi orbifold, loosely speaking parametrizing the position of $\L$ (see \eqref{eq:L-geodesic}) in $\V \otimes \RR$;
        \item a CM point in the modular surface corresponding to the class of the definite quadratic form $\qL = \rest{Q}{\LZ}$ (see \eqref{eq:CM-point-qL});
        \item a periodic geodesic in the modular surface corresponding to the class of the indefinite quadratic form $\qoL = \rest{Q}{\oLZ}$ (see \eqref{eq:geodesic-qoL});
        \item a periodic geodesic in the Bianchi orbifold corresponding to a certain binary form over $\ZZi$, loosely speaking parametrizing the Klein vector (see \eqref{eq:geodesic-qoKL}).
    \end{enumerate}
    
    \begin{remark}
        In what follows, we will identify $\SpinV(\RR)$ with $\SL_2(\CC)$ and think of $\SL_2(\ZZi)$ as a lattice in the real Lie group $\SpinV(\RR)$.
    \end{remark}
\subsection{The geodesic corresponding to the \enquote{position} of \texorpdfstring{$L$}{L}}\label{sec:geodesic-KL}
    We fix the \enquote{reference plane} $\Lo = \langle e_1,e_2\rangle_{\QQ}\in\P_1$. 
    Let $D>0$ be an odd square-free integer.
    Given a plane $\L\in\P_{D}$ choose a rotation $k_\L\in\SpinV(\RR)$ with $k_L.(\Lo(\RR)) = \L(\RR)$; such a rotation exists by Witt's theorem, see e.g. \cite[p.21]{cassels2008rational}. 
    Define
    \begin{align}\label{eq:L-geodesic}
         [\L] \defeq \SL_2(\ZZi)\Xi_1\HH_\L(\RR)k_\L = \SL_2(\ZZi)\Xi_1 k_\L A_\CC \subseteq \X_{\QQi},
    \end{align}
    where we used $\HH_{\Lo}(\RR) = A_\CC$ (cf.~Example~\ref{exp:stab L0}) and where
    \begin{align}\label{eq:xi}
        \Xi_1 = \sqrt{1+i}\begin{pmatrix}1+i&1\\&1\end{pmatrix}^{-1}.
    \end{align}
    Note that $[\L]$ does not depend on the choice of $k_\L$.
    As $D$ and $-D$ are not squares in $\QQ$, the torus $\HH_\L$ is $\QQ$-anisotropic and hence $[\L]$ is compact. 
    Indeed, by the Borel--Harish-Chandra theorem we know that
    \[ \SL_2(\ZZi) \cap \Xi_1\HH_\L(\RR) \Xi_1^{-1} =\Xi_1(\SpinV(\ZZ)\cap\HH_\L(\RR))\Xi_1^{-1} \]
    is a lattice in the Lie group $\Xi_1\HH_\L(\RR)\Xi_1^{-1}$.

\subsection{The shape of \texorpdfstring{$\LZ$}{L(Z)}}\label{sec:CM-point-qL}
    Recall that we represent vectors in $\V = \QQ^4$ as row vectors. For $\alpha\in\SpinV$ we denote by $[\Adj_\alpha]\in \SO_Q$ the matrix such that for all $v \in \V$
    \[ \Adj_{\alpha}(v) = v[\Adj_{\alpha}]^{-1} = v[\Adj_{\alpha^{-1}}]. \]
    For $\L \in \P_D$ we choose a matrix $\BT{\L}\in\SL_4(\ZZ)$ such that the first and second row $v_1,v_2$ form an oriented basis of $\LZ$.
    Let $\qL(x,y) = \Q(xv_1+yv_2)$.
    With $k_\L$ as chosen above, we have that
	\[ \BT{\L}[\Adj_{k_\L}] \in \left\{\begin{pmatrix}*&*&0&0\\ *&*&0&0\\ *&*&*&*\\ *&*&*&*\end{pmatrix}\right\} \subseteq \SL_4(\RR). \]
	We define the shape $[\LZ]$ of the oriented rational plane $\L$ as the $\SO_{2}(\RR)$-orbit of the (renormalized) upper-left $(2\times2)$-matrix above in $\X$, that is
    \begin{align}\label{eq:CM-point-qL}
        [\LZ] \defeq \SL_2(\ZZ)\SO_{\qL}(\RR)\M_\L = \SL_2(\ZZ)\M_\L\SO_2(\RR) \subseteq \X,
    \end{align}
	where $M_\L \defeq D^{-1/4}\pi_{\upperleft}\left(\BT{\L}[\Adj_{k_\L}]\right)$ and $\pi_{\upperleft}$ is the described projection.
    Changing $\BT{\L}$ alters $M_\L$ by an element of $\SL_2(\ZZ)$ on the left.
    Moreover, $[\LZ]$ is independent of the choices of $k_\L$ and $\BT{\L}$.
    
    For all $h\in\HH_\L$ we have
	\[ \BT{\L}[\Adj_h]\BT{\L}^{-1} \in \left\{\begin{pmatrix}*&*&0&0\\ *&*&0&0\\ *&*&*&*\\ *&*&*&*\end{pmatrix}\right\}. \]
    By definition, all elements $h\in\SpinV$ preserve the quadratic form $\Q$.
    We define the morphism $\theta_{\L}\colon\HH_\L \to\HH_{\qL}\defeq\SO_{\qL}$ by
	\begin{align}\label{eq:def-thetaL}
	   \theta_{\L}(h) \defeq \pi_{\upperleft}(\BT{\L}[\Adj_h]\BT{\L}^{-1}). 
	\end{align}
	Note that the determinant of $\theta_{\L}(h)$ is indeed $1$ as $\HH_\L$ is connected. The group $\HH_{\qL}$ and the homomorphism $\theta_{\L}$ depend on the choice of oriented basis of $\L$ implicit in $\BT{\L}$ (but not on the other rows of $\BT{\L}$).
    The morphism $\theta_{\L}$ should be thought of as describing the action of $\HH_{\L}$ on the plane $\L$.
    
\subsection{The periodic geodesic corresponding to \texorpdfstring{$\oLZ$}{orthogonal L(Z)}}\label{sec:geodesic-qoL}
    Let $\L \in \P_D$.
    We construct the shape $[\oLZ]$ \textemdash~a periodic geodesic attached to the orthogonal complement of $\L$ \textemdash~in an analogous way. Note that $k_\L.(\oLo(\RR)) = \oL(\RR)$. 
    We choose a matrix $\BTo{\oL}\in\SL_4(\ZZ)$ such that the third and fourth row $v_3,v_4$ form an oriented basis of $\oLZ$. 
    Let $\qoL(x,y) = \Q(x v_3 + y v_4)$.
    Fix $\delta_3 \in \SL_2(\RR)$ with $\delta_3^{-1}\SO_{1,1}(\RR)\delta_3 = A$.
    We define the geodesic attached to $\oL$ as the $A$-orbit
    \begin{align}\label{eq:geodesic-qoL}
	    [\oLZ] \defeq \SL_2(\ZZ)\SO_{\qoL}\!(\RR)M_{\oL}\delta_3 = \SL_2(\ZZ)M_{\oL}\delta_3A \subseteq \X,
	\end{align} 
	where $M_{\oL} \defeq D^{-1/4}\pi_{\lowerright}\left(\BTo{\oL}[\Adj_{k_\L}]\right)$ and $\pi_{\lowerright}$ denotes the projection to the lower-right $(2\times2)$-matrix. 
    The geodesic is periodic since $\SO_{\qoL}\!(\RR)$ is a $\QQ$-anisotropic torus and is independent of the choices of $k_\L$ and $\BTo{\oL}$. 
    As before, we can define the $\QQ$-morphism 
    \begin{align*}
        \theta_{\oL}\colon\HH_\L \to \HH_{\qoL}\defeq\SO_{\qoL}
    \end{align*}
    via
	\[ \theta_{\oL}(h) \defeq \pi_{\lowerright}(\BTo{\oL}[\Adj_h]\BTo{\oL}^{-1}). \]
    
    Changing $\delta_3$ by a Weyl element if necessary, we may assume that $\delta_3^{-1}\theta_{L_0^\perp}(a_s)\delta_3 = a_{2s}$. Indeed, $a_s$ acts with eigenvalues $e^s,e^{-s}$ on $L_0$ as is clear from Example~\ref{exp:stab L0}.

\subsection{A geodesic parametrizing \texorpdfstring{$\oKLZ$}{orthogonal Klein L(Z)}}\label{sec:geodesic-qoKL} 
    Recall that $\W$ is the subspace of the even Clifford algebra with trace zero over the center $\QQ(i)$ (where $\W$ has dimension $3$ over $\QQ(i)$).
    We represent elements in $\W$ as row vectors with respect to the basis $\{f_1,f_2,f_3\}$ given in \eqref{eq:fi basis}.
    Recall from \eqref{eq:spin-to-SOQF} that the action of $\SpinV(\QQ)$ on $\W$ induces, with this choice of basis, a morphism
    \begin{align*}
        \alpha\in\SpinV \mapsto [[\Adj_\alpha]] \in \Res{\QQi}(\SO_{\QF}) \subset \Res{\QQi}(\GL_3).
    \end{align*}
    In particular, for $\alpha\in\SpinV(\QQ)$ we view $[[\Adj_\alpha]]$ as $(3\times3)$-matrix over $\QQi$ acting on the vector space $\W$ by multiplying row vectors from the right.

    Let $\L \in \mathcal{P}_D$ and let $\BToK{\L}\in\SL_3(\ZZi)$ such that the second and third row form a basis of the oriented orthogonal lattice $\oKLZ$ defined in \eqref{eq:Klein-orth-lattice}.
    Since $\Adj_{k_\L}^{-1}(\L(\RR)) = \Lo(\RR)$ we get by equivariance of the Klein map  that $\Adj_{k_\L}^{-1}(\oKL) = \oKLo = \orth{(2f_1)}$. Hence,
	\begin{align}\label{eq:str-lower-2x2} 
		\BToK{\L}[[\Adj_{k_\L}]]  \in \left\{\begin{pmatrix}*&*&*\\0&*&*\\0&*&*\end{pmatrix}\right\} \subseteq \SL_3(\CC).
	\end{align}
    Let $\qoKL$ be the binary quadratic form over $\ZZi$ obtained by restricting the quadratic form $\QF$ to $\oKLZ$ and using the above basis.
    We define $\HH_{\qoKL}$ as the restriction of scalars from $\QQi$ to $\QQ$ of the $\QQi$-group
    \[ \big\{g\in\SL_2\st g.\qoKL = \qoKL\}. \]    
    In particular, since $\Klein(\Lo) = 2f_1$ we get $\HH_{\qoKLo}=\Res{\QQi}(\SO_2)$ using the oriented basis $\{f_2,f_3\}$ of $\Lambda_{L_0}$. 
    Fix $\delta_4 \in \SL_2(\CC)$ with $\delta_4^{-1}\SO_2(\CC)\delta_4 = A_\CC$.
    Analogously as above, we define the shape $[\oKLZ]$ as the $\SO_2(\CC)=\HH_{\qoKLo}\!(\RR)$-orbit of the (renormalized) lower-right $(2\times2)$-matrix in $\X_{\QQi}$, that is
    \begin{align}\label{eq:geodesic-qoKL}
        [\oKLZ] \defeq  \SL_2(\ZZi)\HH_{\qoKL}(\RR)N_\L\delta_4 = \SL_2(\ZZi)N_\L\delta_4 A_\CC\subseteq \X_{\QQi},
    \end{align}
	where $N_L \defeq D^{-1/4}\pi_{\slowerright}\left(\BToK{\L}[[\Adj_{k_\L}]]\right)$ and $\pi_{\slowerright}$ denotes the described projection. By Lemmas~\ref{lem:disc-of-orth} and \ref{lem:length-of-Klein-vector}, we have $\disc(\Lambda_\L) = \QF(\prKlein(\L)) = D >0$. 
    Also,
    \begin{align*}
        \disc(\Lambda_\L) &= \det\left(\pi_{\slowerright}\left(\BToK{\L}\BToK{\L}^t\right)\right)\\
        &= \det\left(\pi_{\slowerright}\left(\BToK{\L}[[\Adj_{k_\L}]][[\Adj_{k_\L}]]^t\BToK{\L}^t\right)\right) \\
        &= \det\left(\pi_{\slowerright}\left(\BToK{\L}[[\Adj_{k_\L}]]\right)\pi_{\slowerright}\left(\BToK{\L}[[\Adj_{k_\L}]]\right)^t\right) = \det\left(\pi_{\slowerright}\left(\BToK{\L}[[\Adj_{k_\L}]]\right)\right)^2
    \end{align*}
    where the first equality follows from $\QF$ being sum of three squares, the second equality follows from $[[\Adj_{k_\L}]]\in\Res{\QQi}(\SO_{\QF})$ and the second to last equality by \eqref{eq:str-lower-2x2}. Thus, we have $\det(\pi_{\slowerright}(\BToK{\L}[[\Adj_{k_\L}]])) = D^{1/2}$, so that the rescaling factor $D^{-1/4}$ in \eqref{eq:geodesic-qoKL} ensures that we project to an element in $\SL_2(\CC)$.    
    Note that $[\oKLZ]$ is compact as $\HH_{\qoKL}$ is $\QQ$-anisotropic because $D$ and $-D$ are not squares. Again by the Borel--Harish-Chandra theorem, we have that
    \[ \SL_2(\ZZi) \cap \HH_{\qoKL}(\RR) = \Res{\QQi}(\SL_2)(\ZZ)\cap\HH_{\qoKL}(\RR)\]
    is a lattice in $\HH_{\qoKL}(\RR)$.
    
    We define a $\QQ$-morphism $\theta_{\oKLZ}\colon\HH_{\L}\to\HH_{\qoKL}$  by
	\[ \theta_{\oKLZ}(h) \defeq \pi_{\slowerright}(\BToK{\L}[[\Adj_h]]\BToK{\L}^{-1}). \]
    Changing $\delta_4$ by a Weyl element if necessary, we may assume that $\delta_4^{-1}\theta_{\oLoZ}(a_s)\delta_4 = a_{2s}$, as before.

\subsection{From planes to arithmetic equidistribution}\label{sec:planes-to-equi}

    So far, we have defined for each $\L\in \P_D$ a CM point or periodic geodesics in each of the factors of
    \[ \XX \defeq \lquot{\GG(\ZZ)}{\GG(\RR)} = \X_{\QQi}\times\X\times\X\times\X_{\QQi}.  \]
    We now lift these collections in the factors to a single \enquote{coupled} collection on $\XX$ dictated by the various orthogonal complement constructions above.
    
    As discussed in the beginning of \Cref{sec:shapes} we define the $\QQ$-morphism $\Theta_{\L}\colon\HH_\L\to\GG$ by
	\[ \Theta_{\L}(h) \defeq (h,\theta_{\L}(h), \theta_{\oL}(h), \theta_{\oKLZ}(h)). \]
	The image is a two-dimensional $\QQ$-anisotropic torus $\TT_\L \defeq \Theta_{\L}(\HH_\L)$ defined over $\QQ$. 
    Further, in order to align all planes with the reference plane $\Lo$ we define
    \begin{align}\label{eq:shift-element}
        y_\L \defeq (k_\L, M_{\L},M_{\oL}\delta_3, N_{\L}\delta_4)\in\GG(\RR),
    \end{align} 
    which satisfies $y_\L^{-1}\TT_\L(\RR)y_\L \subseteq A_\CC\times\SO_2(\RR)\times A\times A_\CC$. Finally, let
    \begin{align}\label{eq:Real-Orbit}
        \GG(\ZZ)\Zshift\TT_{\L}(\RR)y_\L \subseteq \XX,
    \end{align}
    where $\Zshift \defeq (\Xi_1,e,e,e)$, see \eqref{eq:xi}. The above orbit is naturally equipped with a $y_\L^{-1}\TT_\L(\RR)y_\L$-invariant probability measure.
    This measure relates to the arithmetic coupling of the geodesics and CM points discussed in \S\ref{sec:geodesic-KL}--\ref{sec:geodesic-qoKL}, as we now describe.
    
     We want to define a measure $\nu^{joint}_{\L}$ on $\Tan(\Y_{\QQi})\times\Y\times\Tan(Y)\times\Tan(\Y_{\QQi})$ for a plane $\L\in\P_{D}$. To this end, we first note the following.

    \begin{lemma}\label{lem:lengths}
    For $L \in \P_D$ denote by $\ell_L$ the length of the projection of $[L]$ to $\Tan(\Y_{\QQi})$ (a periodic geodesic).
    Then the following hold:
    \begin{enumerate}
        \item The length $\ell_{\oLZ}$ of the geodesic $[\oLZ]$ is either $\ell_L$ or $2\ell_L$.
        \item The length $\ell_{\oKLZ}$ of the projection of $[\oKLZ]$ to $\Tan(\Y_{\QQi})$ is $\ell_L$.
    \end{enumerate}
    Moreover, these lengths depend only on the square-free integer $D$.
    \end{lemma}

    In view of the lemma, we set $\ell_{D} \defeq \ell_\L$ for some $L\in \P_D$.

    \begin{proof}  
    Let $U_D$ (resp.~$U_{i,D}$) denote the unit group of the ring of integers of $\QQ(\sqrt{D})$ (resp.~$\QQ(i,\sqrt{D})$) and let $\mathcal{F}=\{\pm1,\pm i\}$ be the group of roots of unity in $\QQ(i,\sqrt{D})$.
    By Dirichlet's unit theorem, the free part of the groups $U_D$ and $U_{i,D}$ has rank~$1$.
    As we assume that $D$ is square-free, both of the lengths $\ell_{\oKLZ}$ and $\ell_L$ are equal to  $\log \abs{\eta}$, where $\eta\in U_{i,D}$ has minimal absolute value and has relative norm $\Norm_{\QQ(i,\sqrt{D})/\QQi}(\eta) = 1$. Similarly, $\ell_{\oLZ}$ is $\log \varepsilon$ where $\varepsilon>1$ is a unit in $U_{D}$ with  $\Norm_{\QQ(\sqrt{D})/\QQ}(\varepsilon) = 1$ and of minimal size.
    So these lengths indeed depend only on $D$.
    For the first assertion, note that $\varepsilon\in U_{i,D}$ and $\Norm_{\QQ(i,\sqrt{D})/\QQ(i)}(\varepsilon)=\Norm_{\QQ(\sqrt{D})/\QQ}(\varepsilon)=1$ and therefore $\eta^k\in\mathcal{F}\varepsilon$ for some $k\in \NN$. If $k=1$ we are done. We have that $\varepsilon':=\Norm_{\QQ(i,\sqrt{D})/\QQ(\sqrt{D})}(\eta)\in U_D$ has $\Norm_{\QQ(\sqrt{D})/\QQ}(\varepsilon')=\Norm_{\QQ(i,\sqrt{D})/\QQ}(\eta) = 1$ and is of size $\abs{\eta}^2$. From the minimality of $\varepsilon=\abs{\eta}^k$ we get $k\leq 2$.
    \end{proof}

    Let $\TT^{pt}_\L = \Theta_\L(\HH_\L^{pt})$ where $\HH^{pt}_{\L}$ is the pointwise stabilizer group (see \eqref{eq:defHLpt}). 
    Define the compact group $\KR \defeq A_\CC^{1}\times\SO_2(\RR)\times\{1\}\times A_\CC^{1}$.
    \begin{lemma}\label{lem:torus-real-compact}
        The following hold:
        \begin{enumerate}[noitemsep]
            \item $\HH_\L(\RR) = \HH^{pt}_\L(\RR)\HH^{pt}_{\oL}(\RR)$ and in particular $\TT_\L(\RR) = \TT^{pt}_{\L}(\RR)\TT^{pt}_{\oL}(\RR)$.\label{eq:torus-real-compact-1}
            \item $y_{\L}^{-1}\TT^{pt}_{\oL}(\RR)y_\L \subseteq \KR$\label{eq:torus-real-compact-2}.
            \item $y_{\L}^{-1}\TT^{pt}_{\L}(\RR)y_\L = \{(a_s,1,a_{2s},a_{2s})\st s\in\RR\} \subseteq A\times\{1\}\times A\times A$. \label{eq:torus-real-compact-3}
            \item $y_{\L}^{-1}\TT_\L(\RR)y_\L \cap \KR = y_{\L}^{-1}\TT^{pt}_{\oL}(\RR)y_\L$.\label{eq:torus-real-compact-4}
        \end{enumerate}
    \end{lemma}
    \begin{proof}
        This is a direct consequence of earlier observations.
    \end{proof}
    For any $\L \in \P_D$ we define $\nu^{joint}_\L$ on $\tquot{\XX}{\KR} = \Tan(\Y_{\QQi})\times\Y\times\Tan(Y)\times\Tan(\Y_{\QQi})$ 
    through
    \begin{align}\label{eq:def-of-msr}
        \nu^{joint}_{\L}(\varphi) \defeq \frac{1}{\ell_{D}}\int_{0}^{\ell_{D}}\varphi\big(z_\L (a_s,1,a_{2s},a_{2s}) \KR\big)\dd{s},
    \end{align}
    for any continuous compactly supported function $\varphi$ on $\tquot{\XX}{\KR}$.
    Here, $z_\L \defeq\GG(\ZZ)\Zshift y_\L$ is the \enquote{starting point}.

    \begin{lemma}
        The measure $\nu^{joint}_\L$ depends only on the $\SpinV(\ZZ)$-equivalence class~of~$\L$.
    \end{lemma}

    \begin{proof}
        First, note that we have already established that the starting point $z_\L$ does not depend on $\BT{\L}$, $\BTo{\oL}$, and $\BToK{\L}$. 
        Moreover, changing $k_\L$ alters $z_\L$ by an element of $\TT_{\Lo}(\RR)$ on the right and hence $\nu^{joint}_\L$ depends only on $\L$.
        For $\gamma\in\SpinV(\ZZ)$ and $\L' = \Adj_{\gamma}(\L)$ notice that we may choose $k_{\L'} = \gamma k_\L$, $\BT{\L'} = \BT{\L}[\Adj_{\gamma}]^{-1}$, $\BTo{\orth{\L'}} = \BTo{\oL}[\Adj_{\gamma}]^{-1}$ as well as $\BToK{\L'} = \BToK{\L}[[\Adj_{\gamma}]]^{-1}$. Since, for example,
        \[ M_{\L'} = D^{-1/4}\pi_{\upperleft}(\BT{\L'}[\Adj_{k_{\L'}}]) = D^{-1/4}\pi_{\upperleft}(\BT{\L}[\Adj_\gamma]^{-1}[\Adj_{\gamma}][\Adj_{k_\L}]) = M_{\L}, \]
        it follows that $z_{\L'} = z_{\L}$ and hence the lemma follows.
    \end{proof}

   Lastly, we set
    \[ \nu^{joint}_{\D} \defeq \frac{1}{\abs{\J_{D}}}\sum_{\SpinV(\ZZ)\L\in\J_{D}}\nu^{joint}_\L. \]
    

\section{The Dynamical Theorem}\label{sec:dynamical-theorem}

	In this section we formulate the dynamical version \Cref{thm:dynamical-version} of \Cref{thm:equi}. The dynamical version relies on an equidistribution result for certain adelic torus orbits. Before we state the theorem we introduce the adelic setup.

\subsection{Adelic and \texorpdfstring{$p$}{p}-adic setup}

    Let $p$ be an odd prime. We define the adelic, respectively $p$-adic, extensions of our base space $\XX = \tlquot{\GG(\ZZ)}{\GG(\RR)}$ as
	\[ \XX_\Adeles \defeq  \lquot{\GG(\QQ)}{\GG(\Adeles)} \quad\text{and}\quad \XX_p \defeq  \lquot{\GG(\ZZ[\tfrac{1}{p}])}{\GG(\RR \times\QQp)}. \]
    For $1\leq j \leq 4$, we denote by $\XX_j$, $\XX_{\Adeles,j}$ and $\XX_{p,j}$ the $j$-th factor of the spaces $\XX$, $\XX_\Adeles$ and $\XX_p$, respectively. We will write $x= (x_1,x_2,x_3,x_4)$ to denote the components of $x$ in all of these spaces.
    
    To account for the conjugating element between $\SpinV(\ZZ)$ and $\SL_2(\ZZi)$ (see Proposition~\ref{prop:iso-cliff}), we define the maximal compact open subgroup
    \begin{align*}
    \KGf \defeq \Zshift^{-1}\GG(\ZZhat)\Zshift < \GG(\Adeles_f)
    \end{align*}
    Notice that
    $K_f^\GG = \Zshift^{-1}\GG(\ZZ_2)\Zshift\times\prod_{q\neq 2}\GG(\ZZp[q])$.
    Recall from \eqref{eq:defG} that $\GG$ has class number one with respect to $\KGf$. 
    In particular, there is a $\GG(\RR)$-equivariant projection $\XX_{\Adeles} \to \XX$ (resp.~
    $\XX_{\Adeles} \to \XX_p$) defined by taking the quotient with $\KGf$ (resp.~$\KGf\cap\prod_{q\neq p}\GG(\QQp[q])$) from the right. 
    Explicitly, the map $\XX_{\Adeles} \to \XX$ is given by
    \begin{align}\label{eq:pi}
    \pr\colon\XX_\Adeles\to\XX,\ \GG(\QQ)g \mapsto \GG(\ZZ)\Zshift g'_{\infty}
    \end{align}
    where $g'\in\KG$ satisfies $\GG(\QQ)g = \GG(\QQ)g'$.
    Here, we use that $g'_\infty$ is well-defined up to a left translate by an element of $\GG(\QQ)\cap \KGf = \Zshift^{-1}\GG(\ZZ)\Zshift$ where the equality holds because $\Zshift$ normalizes $\GG(\QQ)$.

\subsection{Equidistribution of torus orbits}\label{subsec:planes->measures}

    Let $p$ be an odd prime and let $\L\in\P_{D}$ be an oriented rational plane with $D$ odd and square-free. For the dynamical \namecref{thm:dynamical-version} we will associate a measure $\mu_\L$ on the $p$-adic extension $\XX_p$ to the given oriented rational plane $\L$.
    We consider the shifted adelic torus orbit 
    \begin{align}\label{eq:Adelic-Orbit}
         \GG(\QQ)\TT_\L(\Adeles)y_\L\subseteq\XX_\Adeles
    \end{align}
    (which can be thought of as a lift of \eqref{eq:Real-Orbit}). 
    It is compact because $\TT_\L$ is a two-dimensional $\QQ$-anisotropic torus, since $D$ is not a square. We project this compact adelic torus orbit 
    to $\XX_p$ and denote the image by $Y_\L\subseteq\XX_p$.
    We define the measure $\mu_\L$ on $Y_L$ to be the pushforward of the normalized Haar measure on the orbit \eqref{eq:Adelic-Orbit} under the natural projection to $\XX_p$.
    
    \begin{definition}
    	A sequence $(\L_n)_n$ of oriented rational planes with $\disc(\L_n) = D_n>0$ is $p$-\highlight{admissible} if
        \begin{itemize}
    		\item $\Leg{D_n}{p} = 1$ or $\Leg{-D_n}{p} = 1$ for all $n\geq1$,
    		\item $D_n\to\infty$ as $n\to\infty$, and
    		\item $D_n$ is square-free and odd.\label{itm:square-free}
    	\end{itemize}
    \end{definition}
    
    \begin{remark}\label{rem:sqfree}
    	The assumptions that $D_n\to\infty$ and that $D_n$ is square-free imply that $\QF(\prKlein(\L_n))\to\infty$; this is the main reason why we assume that $D_n$ is square-free.
    \end{remark}

    We are ready to state the dynamical version of \Cref{thm:equi}, whose proof we postpone to \Cref{sec:proof-of-dyn-thm}.
    
    \begin{theorem}[Equidistribution of packets]\label{thm:dynamical-version}
        Let $(\L_n)_n$ be a $p$-admissible sequence of oriented rational planes. Then $\mu_{\L_n} \to m_{\XX_p}$ in the weak-* topology, as $n\to\infty$.
    \end{theorem}

    The congruence condition with respect to the prime $p$ in \Cref{thm:dynamical-version} is assumed to exploit the invariance of the measures $\mu_{\L_n}$ under a higher-rank diagonalizable action.


\subsection{Proof of \texorpdfstring{{\Cref{thm:equi}}}{thm} from \texorpdfstring{{\Cref{thm:dynamical-version}}}{thm}.}\label{sec:arithmetic-result}

    In the following, we fix an oriented rational plane $\L\in\P_{D}$ with $D>0$ square-free and odd.
    By \cite[Thm 5.1]{platonov_algebraic_1994}, the class number of $\TT_\L$ with respect to $\KGf$ 
	\[ \left|\dquot{\TT_\L(\QQ)}{\TT_\L(\Adeles)}{\KT}\right|, \]
	is finite, where $\KTf\defeq \TT_\L(\Adeles_f) \cap\KGf$. In particular, we can decompose the shifted adelic torus orbit
	\begin{align}\label{eq:dec-of-TT}
		\orb_\L \defeq \GG(\QQ)\TT_\L(\Adeles)y_\L = \bigsqcup_{\rho\in R_\L}\GG(\QQ)\rho(\KT)y_\L \eqdef \bigsqcup_{\rho\in R_\L} \orb_{\rho}
	\end{align}
    into finitely many distinct $y_L^{-1}(\KT) y_L$-orbits $\orb_{\rho}$ for a subset $R_\L \subseteq \KG$ of cardinality equal to the class number of $\TT_\L$.
    
    We now construct rational planes from the above adelic torus orbit. 
    For ease of exposition, we write $\L_\gamma=\gamma_1.\L$ for all $\gamma\in\GG(\QQ)$ and all oriented rational planes $\L$.
    
	\begin{proposition}[Generating integral points]\label{prop:integral-pts-in-adelic-torus-orbit}
		Let $g\in\KG$ be such that $\GG(\QQ)g\in\GG(\QQ)\TT_\L(\Adeles)$. 
        If $\gamma\in\GG(\QQ)$ satisfies $\gamma^{-1}g\in\TT_\L(\Adeles)$, the following hold:
		\begin{enumerate}
			\item $\L_\gamma$ is an oriented rational plane with the same discriminant as $\L$.\label{en:planes}
			\item\label{en:shapes} The quadratic form $\gamma_2.\qL$ is integral and equivalent to $\q_{\L_\gamma}$. Moreover,
            \begin{align*}
            \gamma_{2}M_\L \in \SL_2(\ZZ)M_{\L_\gamma}\theta_{\Lo}(k_{\L_\gamma}^{-1}\gamma_{1}k_\L) 
            \end{align*}
            where $M_\L$ is defined in \eqref{eq:CM-point-qL}.
			Analogous statements hold for $\qoL$ and~$\qoKL$.
		\end{enumerate}
	\end{proposition}
    It follows that for any $\rho \in R_\L$, we obtain an oriented plane $\L_{\rho}\defeq \L_{\gamma}$ which is an element of $\P_{D}$ for a choice of $\gamma\in\GG(\QQ)$ satisfying $\gamma^{-1}\rho\in\TT_\L(\Adeles)$.
	\begin{proof}
		By assumption, we can write $g=\gamma t$ with $\gamma\in\GG(\QQ)$ and $t=\Theta_\L(h)\in\TT_\L(\Adeles)$, for some $h\in\HH_\L(\Adeles)$. 
        
        Note that for every prime $p$ and for $v_1,v_2$ an oriented $\ZZ$-basis of $\L$ we have
        \begin{align*}
            \varV[](\ZZ_p) \ni g_{1,p}.(v_1 \wedge v_2) = \gamma_1h_{p}.(v_1\wedge v_2) = \gamma_1.(v_1\wedge v_2) \in \varV[](\QQ)
        \end{align*}
        using $g_{1,p} = \gamma_1h_{p}$. Thus, $\gamma_1.(v_1\wedge v_2)$ is a primitive integral wedge and so necessarily equal to the integral wedge corresponding to $\L_\gamma$.
        In particular, $\disc(\L) = \disc(\L_{\gamma})$ and $\Klein(\L_\gamma) = \Adj_{\gamma_1}(\Klein(\L))$ by equivariance of the Klein map. 
		
		To prove \eqref{en:shapes}, we first note that $\gamma_1 k_\L\in\SpinV(\RR)$ satisfies $\gamma_1 k_\L.\Lo(\RR) = \L_\gamma(\RR)$. Moreover, the rational matrix
		$\BT{\L}[\Adj_{\gamma_1}]^{-1}\BT{\L_\gamma}^{-1}$ maps the basis (row)  vectors $e_1$ and $e_2$ of the plane $\Lo$ to some vectors in $\Lo$ and is thus of the block form
		\[ \BT{\L}[\Adj_{\gamma_1}]^{-1}\BT{\L_\gamma}^{-1} \in \left\{\begin{pmatrix}
			\begin{array}{cccc}
				* & * & 0 & 0 \\
				* & * & 0 & 0 \\
				* & * & * & * \\
				* & * & * & * \\
			\end{array}
		\end{pmatrix}\right\}. \]
		For $A \defeq \pi_{\upperleft}(\BT{\L}[\Adj_{\gamma_1}]^{-1}\BT{\L_\gamma}^{-1}) \in\GL_2(\QQ)$ we get
		\begin{align*}
			\gamma_2.\qL(x,y)&= \qL((x,y)\gamma_2) = \Q(((x,y)\gamma_2,0,0)\BT{\L})\\
			&= \Q(((x,y)\gamma_2,0,0)\BT{\L}[\Adj_{\gamma_1}]^{-1}\BT{\L_\gamma}^{-1}\BT{\L_\gamma}) = \gamma_2A.q_{\L_\gamma}(x,y).
		\end{align*}
		We now show that $\gamma_2A\in\SL_2(\ZZ)$ which implies that $\gamma_2.\qL$ is indeed integral and equivalent to $\q_{\L_\gamma}$. Recall that $\gamma = gt^{-1} = g\Theta_\L(h)^{-1}$. Therefore, we have
		\begin{align}
            \begin{split}
    			\GL_2(\QQ)\ni \gamma_2A &= g_2t_2^{-1}\pi_{\upperleft}(\BT{\L}[\Adj_{g_1 h^{-1}}]^{-1}\BT{\L_\gamma}^{-1})\\
    			&= g_2t_2^{-1}\pi_{\upperleft}(\BT{\L}[\Adj_{h}]\BT{\L}^{-1})\pi_{\upperleft}(\BT{\L}[\Adj_{g_1}]^{-1}\BT{\L_\gamma}^{-1})\\
    			&=g_2\pi_{\upperleft}(\BT{\L}[\Adj_{g_1}]^{-1}\BT{\L_\gamma}^{-1}) \in \GL_2(\RR\times\ZZhat),
            \end{split}\label{eq:integral-matrix}
		\end{align}
		where the last equality follows from $t_2 = \theta_\L(h) = \pi_{\upperleft}(\BT{\L}[\Adj_{h}]\BT{\L}^{-1})$. Hence, we have $\gamma_2A \in \GL_2(\QQ)\cap\GL_2(\RR\times\ZZhat) = \GL_2(\ZZ)$.
        Since $\gamma_1.(\Klein(\L)) = \Klein(\L_{\gamma})$, the orientations of $\L$ and $\L_\gamma$ match, which is equivalent to $\det(A) > 0$ and so $\gamma_2A\in\SL_2(\ZZ)$. 
        This proves the first claim in \eqref{en:shapes}.
        
        By a straightforward calculation using the definition of $A$
        \begin{align}\label{eq:2nd-comp}
            \begin{split}
                A M_{\L_\gamma}&\theta_{\Lo}(k_{\L_\gamma}^{-1}\gamma_{1}k_\L) \\
                &= \pi_{\upperleft}(\BT{\L}[\Adj_{\gamma_1}]^{-1}\BT{\L_\gamma}^{-1})M_{\L_\gamma}\theta_{\Lo}(k_{\L_\gamma}^{-1}\gamma_{1}k_\L)\\
                &=D^{-1/4}\pi_{\upperleft}(\BT{\L}[\Adj_{\gamma_1}]^{-1}\BT{\L_\gamma}^{-1}\BT{\L_{\gamma}}[\Adj_{k_{\L_{\gamma}}}])\theta_{\Lo}(k_{\L_\gamma}^{-1}\gamma_{1}k_\L)\\
                & = D^{-1/4}\pi_{\upperleft}(\BT{\L}[\Adj_{\gamma_1}]^{-1}[\Adj_{k_{\L_{\gamma}}}])\pi_{\upperleft}([\Adj_{k_{\L_{\gamma}}}]^{-1}[\Adj_{\gamma_1}][\Adj_{k_{\L}}])\\
                &= D^{-1/4}\pi_{\upperleft}(\BT{\L}[\Adj_{k_{\L}}]) = M_{\L}.
            \end{split}
        \end{align}
        Multiplying both sides with $\gamma_2$ and using that $\gamma_2A\in\SL_2(\ZZ)$ gives the second claim in \eqref{en:shapes}.
        The statement for $\gamma_3.\qoL$ is analogous. 
        
        It remains to prove the statement about $\gamma_4.\qoKL$. The matrix 
		\[ \BToK{\L}[[\Adj_{\gamma_1}]]^{-1}\BToK{\L_\gamma}^{-1}\in \SL_3(\QQi) \]
		maps $f_2$ and $f_3$, which span $\oKLoZ$, again to elements in $\oKLoZ$, hence
		\begin{align}\label{eq:BT-matrix}
    		\BToK{\L}[[\Adj_{\gamma_1}]]^{-1}\BToK{\L_\gamma}^{-1} \in \left\{\begin{pmatrix}
    			\begin{array}{ccc}
    				* & * & * \\
    				0 & * & * \\
    				0 & * & * \\
    			\end{array}
    		\end{pmatrix}\right\}.
        \end{align}
		As before, setting $B \defeq \pi_{\slowerright}(\BToK{\L}[[\Adj_{\gamma_1}]]^{-1}\BToK{\L_\gamma}^{-1})\in\GL_2(\QQi)$ we obtain that $\gamma_4 .\qoKL(x,y) = \gamma_4B.\q_{\Lambda_{\L_\gamma}}(x,y)$.
		We show that $\gamma_4B\in\SL_2(\ZZi)$, which implies that $\gamma_4.\qoKL$ is $\ZZi$-integral and equivalent to $q_{\Lambda_{\L_\gamma}}$. Using $\gamma = gt^{-1} = g\Theta_\L(h)^{-1}$ we get
		\begin{align}
            \begin{split}\label{eq:integral-matrix-2}
    			\GL_2(\QQi) \ni \gamma_4B &= \gamma_4\pi_{\slowerright}(\BToK{\L}[[\Adj_{\gamma_1}]]^{-1}\BToK{\L_\gamma}^{-1})\\
    			&=g_4t_4^{-1}\pi_{\slowerright}(\BToK{\L}[[\Adj_{g_1 h^{-1}}]]^{-1}\BToK{\L_\gamma}^{-1})\\
    			&=g_4t_4^{-1}\pi_{\slowerright}(\BToK{\L}[[\Adj_{h}]]\BToK{\L}^{-1})\pi_{\slowerright}(\BToK{\L}[[\Adj_{g_1}]]^{-1}\BToK{\L_\gamma}^{-1})\\
    			&=g_4\pi_{\slowerright}(\BToK{\L}[[\Adj_{g_1}]]^{-1}\BToK{\L_\gamma}^{-1}) \in \GL_2(\CC\times(\ZZhat \otimes \ZZi)),
            \end{split}
		\end{align}
		where the last equality holds as $t_4 = \pi_{\slowerright}(\BToK{\L}[[\Adj_{h}]]\BToK{\L}^{-1})$. Therefore,
		\[ \gamma_4B \in \GL_2(\QQi)\cap \GL_2(\CC\times(\ZZhat \otimes \ZZi)) = \GL_2(\ZZi). \]
        Now recall that the second and third rows of $\BToK{\L}$ and $\BToK{\L_\gamma}$ form an oriented basis of the oriented lattices $\oKLZ$ and $\Lambda_{\L_\gamma}$, respectively. Since $\Adj_{\gamma_1}(\prKlein(\L)) = \Adj_{\gamma_1}(\Klein(\L)/2) = \Klein(\L_{\gamma})/2 = \prKlein(\L_{\gamma})$ as we assume that $D>0$ is square-free and odd, and $\Adj_{\gamma_1}$ preserves orientation, the second and third row of $\BToK{\L}[[\Adj_{\gamma_1}]]^{-1}$ give an oriented basis of $\orth{\Klein(\L_\gamma)}$ as well. Since two oriented bases of $\orth{\Klein(\L_\gamma)}$ differ by an element of $\GL_2(\QQi)$ with determinant in $\QQ_{>0}$ only, this implies that $B$ has determinant $1$. 
        One deduces that
        \begin{align}\label{eq:4th-comp}
            \gamma_{4}N_{\L} \in \SL_2(\ZZi)N_{\L_\gamma}\theta_{\oKLoZ}(k_{\L_\gamma}^{-1}\gamma_{1}k_\L),
        \end{align}
        by a straightforward calculation similar to \eqref{eq:2nd-comp}.
	\end{proof}

    We denote by $\J(\L) = \{\SpinV(\ZZ)\L_{\rho}\st\rho\in R_\L\}\subset \J_{\disc(\L)}$ and call it the \highlight{collection associated to the oriented plane $\L$}. 
    Note that the set $\{\L_{\rho}\st\rho\in R_\L\}$ depends on the choice of representatives $R_\L$, but the collection $\J(\L)$ does not. 
    Next, we make a consistent choice of rotations $k_{\L'}$ (see \S\ref{sec:geodesic-KL}) for all $L'$ with $\Spin_Q(\ZZ)\L'\in \J(\L)$ so that the choice agrees, in particular, with the choice of representatives $R_\L$. 
    For any $\rho \in R_\L$ choose $k_{\L_\rho}= \rho_{\infty,1} k_\L$. For $\gamma \in \Spin_Q(\ZZ)$ set $k_{\gamma.\L_\rho} = \gamma k_{\L_\rho}$.
    Recall the definition of $y_{L'},z_{L'}$ from \S\ref{sec:planes-to-equi} and of $\pr$ from \eqref{eq:pi}.
    
    \begin{lemma}\label{lem:choice-of-reps} 
        We have for all $\rho \in R_\L$
        \[ \pr(\GG(\QQ)\rho y_\L) = \GG(\ZZ)\Zshift\rho_\infty y_\L = \GG(\ZZ)\Zshift y_{\L_\rho} = z_{\L_{\rho}}. \]
    \end{lemma}
    \begin{proof}
    The first and last equalities are direct consequence of our definitions and we verify the second equality componentwise.
    In the first component we have equality by the choice of $k_{L_\rho}$.
    We prove equality in the second component, the third and fourth are analogous.
    Let $\gamma \in \GG(\QQ)$ with $ t= \gamma^{-1}\rho\in\TT_\L(\Adeles)$ and write $t = \Theta_L(h)$.
    Using Proposition~\ref{prop:integral-pts-in-adelic-torus-orbit}\eqref{en:shapes}, we have
    \begin{align*}
    \SL_2(\ZZ)\rho_{\infty,2}M_{L}
    &= \SL_2(\ZZ) \gamma_2 \theta_L(h_\infty) M_L
    = \SL_2(\ZZ) \gamma_2  M_L \theta_{L_0}(k_L^{-1}h_\infty k_L)\\
    &= \SL_2(\ZZ) M_{L_\gamma} \theta_{L_0}(k_{L_\gamma}^{-1}\gamma_1k_L) \theta_{L_0}(k_L^{-1}h_\infty k_L) = \SL_2(\ZZ) M_{L_\gamma}
    \end{align*}
    where we used that $M_L^{-1}\theta_L(h_\infty) M_L= \theta_{L_0}(k_L^{-1}h_\infty k_L)$ and that $k_{L_\gamma}^{-1}\gamma_1h_\infty k_L$ is trivial by our choices.
    \end{proof}

    We define the map $\prK\colon\XX_\Adeles\to \tquot{\XX}{\KR}$ by postcomposing $\pr$ with the natural projection $\XX\to\tquot{\XX}{\KR}$ where $\KR$ is defined before \Cref{lem:torus-real-compact}. Recall the definition of $\orb_{\L}$ and $\orb_{\rho}$ for $\rho\in R_L$ in \eqref{eq:dec-of-TT}.
    
    \begin{lemma}\label{lem:disjoint-proj}
        Suppose for $\rho,\rho'\in R_\L$ that $\prK(\orb_{\rho})_1 \cap \prK(\orb_{\rho'})_1 \neq \emptyset$. Then $\rho = \rho'$. In particular, $\prK(\orb_{\L}) = \bigsqcup_{\rho\in R_\L} \prK(\orb_{\rho})$ is a disjoint union.
    \end{lemma}
    \begin{proof}
        By definition of $R_\L$ there exist $\gamma,\gamma'\in\GG(\QQ)$ such that $\tau \defeq \gamma\rho\in\TT_\L(\Adeles_f)$ and $\tau' \defeq \gamma'\rho'\in\TT_\L(\Adeles_f)$. By assumption, there exist $t,t'\in\KT$ such that
        \[ \prK(\GG(\QQ)\tau t y_\L)_1 = \prK(\GG(\QQ)\tau' t' y_\L)_1. \]
        Thus, by \Cref{lem:torus-real-compact} and the definition of $\prK$ there exists $\delta\in\GG_1(\QQ)$ such that
        \[ \delta \tau_1 t_1 k_\L \in \tau'_1 t'_1 k_\L (\HH^{pt}_{\Lo}(\RR)\times (\KGf)_1) \]
        which is equivalent to $ \delta \tau_1 t_1 \in \tau'_1 t'_1 (\HH^{pt}_{\L}(\RR)\times (\KGf)_1)$  since $\HH^{pt}_{\L}(\RR) = k_\L\HH^{pt}_{\Lo}(\RR) k_\L^{-1}$. Looking at the infinite place we obtain that $\delta\in\HH_\L(\RR)\cap\GG(\QQ) = \HH_\L(\QQ)$ and so it follows that $\delta \tau_1 t_1\in\HH_\L(\Adeles)$. This implies that
        \[ \delta \tau_1 t_1 \in \tau'_1 t'_1(\HH^{pt}_{\L}(\RR)\times (\KTf)_1). \]
        Recall that the map $\Theta_\L\colon\HH_\L\to\TT_\L$ is defined over $\QQ$ so $\Theta_L(\delta)\in\GG(\QQ)$ and observe further that $\Theta_\L\colon(\KTf)_1\to\KTf$. 
        In particular,
        \[ \Theta_\L(\delta \tau_1 t_1) \in \Theta_\L(\tau'_1 t'_1(\HH^{pt}_{\L}(\RR)\times (\KTf)_1)) \subseteq  \Theta_\L(\tau'_1 t'_1)(\TT_{\L}(\RR)\times\KTf). \]
        Summing up, we get
        \begin{align*}
            \orb_{\rho} &= \GG(\QQ)\tau(\KT)y_\L
            = \GG(\QQ)\Theta_\L(\delta)\tau t(\KT)y_\L \\
            & = \GG(\QQ)\Theta_{\L}(\delta\tau_1 t_1)(\KT)y_\L = \GG(\QQ)\Theta_{\L}(\tau'_1 t'_1)(\KT)y_\L\\
            & = \GG(\QQ)\tau'(\KT)y_\L = \orb_{\rho'},
        \end{align*}
        which implies that $\rho = \rho'$ and concludes the proof.
    \end{proof}
    	
	Now that we have discussed how to obtain a collection $\J(\L)$ from a single shifted adelic torus orbit $\GG(\QQ)\TT_\L(\Adeles)y_\L$ we can describe the measure $\mu_\L$ (defined in \S \ref{subsec:planes->measures}) on the $p$-adic extension $\XX_p$ in detail.

    \begin{proposition}\label{prop:Projection}
        The pushforward of the measure $\mu_{\L}$ under $\prK$ is equal to
        \begin{align*}
            \frac{1}{\abs{R_\L}}\sum_{\SpinV(\ZZ)\L'\in \J(\L)}\nu^{joint}_{\L'}.
        \end{align*}
    \end{proposition}
    \begin{proof}
        The pushforward of $\mu_L$ is equal to the pushforward of the normalized Haar measure on $\GG(\QQ)\TT_\L(\Adeles)y_\L$ and we may as well study the latter.
        Using \Cref{lem:choice-of-reps}, we observe that for all $\rho\in R_\L$,  $\GG(\QQ)\rho y_\L$ is mapped to the starting point $z_{\L_{\rho}}$. 
        Further, the Haar probability measure on the orbit $\orb_\L = \bigsqcup_{\rho\in R_\L}\orb_{\rho}$ gives equal weights to each subset $\orb_{\rho}$ by commutativity of $\TT_\L(\Adeles)$. By \Cref{lem:disjoint-proj}, $\prK(\orb_{\L}) = \bigsqcup_{\rho\in R_\L}\prK(\orb_{\rho})$ is a disjoint union and by \Cref{lem:torus-real-compact} we have
        \begin{align*}
             \widetilde{\pr}(\orb_\rho) &= \widetilde{\pr}(\GG(\QQ)\rho(\KT)y_\L) = \widetilde{\pr}(\GG(\QQ)\rho y_\L (y_\L^{-1}\TT_\L(\RR)y_\L\times\KTf)) \\
             &= z_{\L_{\rho}} y_\L^{-1}\TT_\L(\RR)y_\L\KR= z_{\L_{\rho}} y_\L^{-1}\TT^{pt}_\L(\RR)y_\L\KR\\
             &= \{z_{\L_{\rho}}(a_s,1,a_{2s},a_{2s})\KR\st s\in\RR\}.
        \end{align*}
        It follows that the pushforward of the Haar measure restricted to $\orb_\rho$ is equal to $\tfrac{1}{\abs{R_\L}}\nu^{joint}_{\L_{\rho}}$. Summing up, we obtain that the pushforward of $\mu_\L$ is given by 
        \begin{align*}
            \frac{1}{\abs{R_\L}}\sum_{\SpinV(\ZZ)\L'\in \J(\L)}\nu^{joint}_{\L'}
        \end{align*} as claimed.
    \end{proof}
    
\begin{proof}[Proof of \Cref{thm:equi} assuming \Cref{thm:dynamical-version}]
        By \Cref{prop:Projection} and \Cref{thm:dynamical-version}, we know that for a $p$-admissible sequence $(\L_n)_n$ of oriented planes the pushforward measures of $\mu_{\L_n}$ given by
		\begin{align*}
			\frac{1}{\abs{R_{\L_n}}}\sum_{\SpinV(\ZZ)\L\in \J(\L_n)}\nu^{joint}_\L
		\end{align*}
        converge to the Haar measure $\msr$ of $\Tan(\Y_{\QQi})\times\Y\times\X\times\Tan(\Y_{\QQi})$. 
        Recall that $\J_D$ as defined in \eqref{eq:Def-of-JD} is a finite set and we can decompose $\J_\D$ into a finite disjoint union of collections,
		$\J_\D = \bigsqcup_{j=1}^{M_D} \J(\L_j)$.
        Thus, $\nu^{joint}_D$ is a finite convex combination of the above measures which individually equidistribute to the Haar measure $\msr$ as $D \to \infty$. The theorem follows.
	\end{proof}


\section{Proof of the Dynamical Theorem \texorpdfstring{\labelcref{thm:dynamical-version}}{thm}}\label{sec:proof-of-dyn-thm}

Let $(\L_n)_n$ be a $p$-admissible sequence of oriented rational planes.
Let $D_n = \disc(L_n)$ and, by passing to a subsequence, assume either that $D_n$ is a non-zero square mod $p$ for all $n$ or that $-D_n$ is a non-zero square mod $p$ for all $n$.
It suffices to show that any convergent subsequence of the sequence of associated measures $\mu_{\L_n}$ converges to the Haar probability measure on $\XX_p$.
To simplify notation, assume that $\mu_{\L_n}$ converges to some measure $\mu$.
We will show that $\mu$ is the Haar probability measure on $\XX_p$.
	
\subsection{Equidistribution in each factor}\label{subsec:EquiEachFac}

    In this section, we discuss the following first step towards \Cref{thm:dynamical-version} mentioned above.

    \begin{proposition}\label{prop:indiv-equi}
    The measure $\mu$ on $\XX_p$ is a joining of the Haar measures on the quotients $\XX_{p,j}$ for the action given by the subgroup $\{(a_s,e,a_{2s},a_{2s})\colon s\in\RR\}$.
	In particular, $\mu$ is a probability measure.  
	\end{proposition}
	
	This \namecref{prop:indiv-equi} should be seen as multiple instances of Duke's theorem \cite{duke_hyperbolic_1988} and its generalizations in various aspects \textemdash~see for instance \cite[\S4]{einsiedler_distribution_2011} for a general discussion and references.
    However, note that the above proposition in each factor is an equidistribution result for subcollections of the usual collections (of CM points, periodic geodesics, etc.) considered e.g.~by Linnik and Duke; see for example the works by Harcos, Michel~\cite{harcos-michel} or M.A.~and Einsiedler~\cite{AE-subcollections}.
    Below we show that these subcollections are large; this is a crucial ingredient for Proposition~\ref{prop:indiv-equi}.

    Let us make the connection to class groups explicit. 

    \begin{lemma}\label{lem:relationtoclassgroups}
    Let $\L\in \P_D$ for $D>0$ square-free and let $K =\QQ(i,\sqrt{D})$.
    Then the finite group 
    \begin{align}\label{eq:class-group}
        \dquot{\HH_\L(\QQ)}{\HH_\L(\Adeles_f)}{\HH_\L(\ZZhat)}
    \end{align}
    factors onto the group of squares in the class group $\mathfrak{cl}_K$ of~$K$. Moreover, the kernel of the factor map is bounded in size independently of $D$.
    \end{lemma}

    \begin{proof}
    We first claim that
    \begin{align}\label{eq:stab-as-alg}
        \HH_\L \simeq \{\alpha\in\Res{K}(\GG_{m,K}) \st \Norm_{K/\QQi}(\alpha) = 1 \}.
    \end{align}
    For the claim, consider the $\QQi$-linear embedding $\iota\colon K=\QQ(i,\sqrt{-D}) \to \Mat[2](\QQi)$ given by $\iota(\sqrt{-D}) = \prKlein(\L)\in W\subseteq \Mat[2](\QQi)$. 
    The embedding $\iota$ realizes the algebraic torus $\Res{K}(\GG_{m,K})$ as a subgroup  of $\Res{\QQi}(\GL_2)$. 
    Notice that under $\iota$ the relative norm $\Norm_{K/\QQi}$ corresponds to $\QF$, that is, the determinant in $\Mat[2](\QQi)$. 
    Since $\HH_\L \subseteq \Res{\QQi}(\SL_2)$ is the centralizer of $\prKlein(\L)$, or equivalently the centralizer of $\iota(K)$, this proves \eqref{eq:stab-as-alg}.

    Recall that the integral structure on $\HH_\L$ is inherited from the integral structure on $\SpinV$ (the latter is given through the standard lattice in $\V$).
    Similarly to \Cref{prop:iso-cliff}, one shows that $\HH_\L(\ZZ_p) = \HH_\L(\QQ_p) \cap \Xi_1^{-1}\Mat[2](\ZZ[i]\otimes \ZZ_p)\Xi_1$ for every prime $p$.
    Define the order $\mathcal{O} = \iota^{-1}(\Xi_1^{-1}\Mat[2](\ZZ[i])\Xi_1)$ in $K$. By definition, $\ZZ[i] \subset \mathcal{O}\subset \mathcal{O}_K$ where $\mathcal{O}_K$ is the ring of integers of $K$. 
    
    As was shown e.g.~at the beginning of the proof of \Cref{prop:iso-cliff} we have
    \begin{align*}
        \prKlein(\L) \in
            \Big\{ \begin{pmatrix}
            x_1 & x_2 \\ x_3 & x_4
            \end{pmatrix}: x_1,x_2,x_3,x_4 \in \ZZ[i],\ x_1+x_4,x_2+x_3 \in 2\ZZ[i]\Big\}.
    \end{align*}
    The latter lattice is contained in $\Xi_1^{-1}\Mat[2](\ZZ[i])\Xi_1$ and therefore $\ZZ[i,\sqrt{D}] \subset \mathcal{O}$. In particular, the index $[\mathcal{O}_K:\mathcal{O}]$ is absolutely bounded and for any prime $p \neq 2$ the local order $\mathcal{O}\otimes \ZZ_p$ is maximal i.e.~$\mathcal{O}\otimes \ZZ_p = \mathcal{O}_K \otimes \ZZ_p$. 
    We let $\widehat{\mathcal{O}} = \mathcal{O}\otimes \widehat{\ZZ}$ and denote by $K^1, \Adeles_K^1, \widehat{\mathcal{O}}^1$ the respective kernels under the relative norm $\Norm_{K/\QQ(i)}$.
    By the definition of $\iota$, it gives an isomorphism
    \begin{align*}
     \dquot{K^{1}}{\Adeles_{K,f}^{1}}{\widehat{\mathcal{O}}^{1}}
     \xlongrightarrow[]{\iota}
     \dquot{\HH_\L(\QQ)}{\HH_\L(\Adeles_f)}{\HH_\L(\ZZhat)},
    \end{align*}
    so it is enough to prove the lemma for the finite group $K^{1} \backslash\Adeles_{K,f}^{1}/\widehat{\mathcal{O}}^{1}$.
    Moreover, as $\mathcal{O} \subset \mathcal{O}_K$ (with absolutely bounded index), the kernel of the natural map
    \begin{align*}
    \dquot{K^{1}}{\Adeles_{K,f}^{1}}{\widehat{\mathcal{O}}^{1}} \to \dquot{K^{1}}{\Adeles_{K,f}^{1}}{\widehat{\mathcal{O}_K}^{1}}
    \end{align*}
    is absolutely bounded in size. Thus, it is enough to consider $K^{1} \backslash\Adeles_{K,f}^{1}/\widehat{\mathcal{O}_K}^{1}$.

    It is a standard consequence of $\mathcal{O}_K$-ideals being locally principal that the quotient $K^{\times} \backslash\Adeles_{K,f}^{\times}/\widehat{\mathcal{O}_K}^{\times}$ of the idele class group of $K$ can be identified with the class group $\mathfrak{cl}_K$ of $K$.
    Note that there is a natural map
    \begin{align*}
    \phi: K^{1} \backslash\Adeles_{K,f}^{1}/\widehat{\mathcal{O}_K}^{1} \to K^{\times} \backslash\Adeles_{K,f}^{\times}/\widehat{\mathcal{O}_K}^{\times}.
    \end{align*}
    The kernel $\mathcal{F}$ is isomorphic to the set of pairs $(t_1,t_2) \in K^\times \times \widehat{\mathcal{O}_K}^{\times}$ with the property $\Norm_{K/\QQ(i)}(t_1t_2)=1$ and up to the action of $K^1 \times \widehat{\mathcal{O}_K}^{1}$.
    For any such $(t_1,t_2)$ we have $\Norm_{K/\QQ(i)}(t_1) = \Norm_{K/\QQ(i)}(t_2)^{-1} \in \QQ(i) \cap (\ZZ[i] \otimes \widehat{\ZZ})^{\times} = \ZZ[i]^\times$.
    Thus, the relative norm injects $\mathcal{F}$ into the group $\ZZ[i]^\times$ which is of order four. 
    
    It remains to show that the image of $\phi$ is the group of squares.
    First, if $s \in \Adeles_{K,f}^{\times}$ then $\frac{1}{\Norm_{K/\QQ(i)}(s)}s^2$ has relative norm $1$ and represents the same class as $s^2$ since $\QQ(i)$ has class number one.
    In particular, the image of $\phi$ contains the squares.
    Conversely, let $t \in \Adeles_{K,f}^{1}$ and for any finite place $v$ of $F=\QQ(i)$ denote the $v$-component of $t$ by $t_v$.
    We may assume $t_v$ is trivial at all but finitely many places $v$.
    If $K_v = K \otimes F_v$ is a field, then Hilbert's Theorem 90 implies that $t_v = \frac{s_v}{\sigma(s_v)}$ for some $s_v \in K_v$ where $\sigma$ is the non-trivial Galois automorphism of $K/F$.
    If $K_v \simeq F_v \times F_v$, then $t_v\simeq (t_{v,1},t_{v,2})$ where $t_{v,1}t_{v,2}=1$ and, thus, $t_v = \frac{s_v}{\sigma(s_v)}$ for $s_v \simeq (t_{v,1},1)$.
    We have thus written $t = \frac{s}{\sigma(s)}$ where $s = (s_v)_v$.
    Using again that $\QQ(i)$ has class number one, we obtain that $s^2 = \Norm_{K/\QQ(i)}(s)t$ represents the same class as $t$.
    This proves the lemma.
    \end{proof}  

    \begin{lemma}\label{lem:large-packets}
        Let $\L \in \P_D$ where $D>0$ is square-free. 
		The images of the maps $\theta_{\L}\colon\HH_\L(\KK)\to\HH_{\qL}(\KK)$ and $\theta_{\oL}\colon\HH_\L(\KK)\to\HH_{\qoL}(\KK)$ contain the set of squares for any field $\KK$ of characteristic $0$.
		In particular, the image of the induced homomorphism
		\begin{align*}
			\dquot{\HH_\L(\QQ)}{\HH_\L(\Adeles_f)}{\HH_\L(\ZZhat)} 
            &\longrightarrow \dquot{\HH_{\qL}(\QQ)}{\HH_{\qL}(\Adeles_f)}{\HH_{\qL}(\ZZhat)}
		\end{align*}
		has index at most $\D^{o(1)}$ and similarly for $q_{\L^\perp}$.
	\end{lemma}
    
	\begin{proof}
        By definition of $\theta_\L$ in \eqref{eq:def-thetaL} together with \Cref{cor:stabilizer-subgroup}, we obtain that $\theta_\L$ is given by the map $\HH_\L(\KK) \to \HH_{\qL}(\KK) = \SO_{\qL}(\KK) \simeq \Spin_{\Q_{|L}}(\KK),~\beta\beta'\mapsto \ClNrm(\beta')\beta^2$. Clearly, if $u=v^2$ for $v\in\Spin_{\Q_{|L}}(\KK)$, then $v\in\HH_\L(\KK)$ maps to $u$, and thus $\theta_\L$ contains the set of squares for any field $\KK$ of characteristic $0$. The argument that the image of $\theta_{\oL}$ contains the set of squares is the same.

        As $D$ is square-free, the group $\HH_{\qL}(\QQ)\backslash\HH_{\qL}(\Adeles_f)/\HH_{\qL}(\ZZhat)$ is isomorphic to the set of squares in the class group $\mathfrak{cl}_{\QQ(\sqrt{-D})}$ \textemdash~see e.g.~\cite[\S7]{wieser_linniks_2019} which is similar to the proof of Lemma~\ref{lem:relationtoclassgroups} above.
        We also note that the $2$-torsion (and hence also the $4$-torsion) of $\mathfrak{cl}_{\QQ(\sqrt{-D})}$ is of size $\D^{o(1)}$ by standard divisor function estimates (see \cite[p.~342]{cassels2008rational}). 
        The $2$-torsion subgroup of the group $\HH_{\qL}(\QQ)\backslash\HH_{\qL}(\Adeles_f)/\HH_{\qL}(\ZZhat)$ is the $4$-torsion subgroup of $\mathfrak{cl}_{\QQ(\sqrt{-D})}$ and, in particular, of size $D^{o(1)}$.
        The argument for $\L^\perp$ is analogous and the lemma follows.
	\end{proof}

	\begin{lemma}\label{lem:image-of-third-morphism}
        Let $\L \in \P_D$ where $D>0$ is square-free and odd.
		The image of the map $\theta_{\oKLZ}\colon\HH_\L(\KK)\to\HH_{\qoKL}(\KK)$ is equal to the set of squares in $\HH_{\qoKL}(\KK)$ for any field $\KK$ of characteristic~$0$.
		In particular, the image of the induced homomorphism
        \begin{align*}
        \dquot{\HH_\L(\QQ)}{\HH_\L(\Adeles_f)}{\HH_\L(\ZZhat)} 
            &\longrightarrow \dquot{\HH_{\qoKL}(\QQ)}{\HH_{\qoKL}(\Adeles_f)}{\HH_{\qoKL}(\ZZhat)}
        \end{align*}
        has index at most $D^{o(1)}$.
	\end{lemma}

	\begin{proof}
	By Lemma~\ref{lem:compl of Klein}, any element of $\orth{\Klein(\L)}$ is a finite sum of elements of the form $xx'$ for $x\in\L(\KK)$ and $x'\in\oL(\KK)$.
    Moreover, any element of $\HH_\L(\KK)$ is of the form $\beta\beta'$, with $\beta\in\sClGr_\L(\KK)$ and $\beta'\in\sClGr_{\oL}(\KK)$, and its conjugation action on $\orth{\Klein(\L)}$ is given by $\beta\beta'xx'(\beta\beta')^{-1} = (\beta\beta')^2xx'$, by \Cref{cor:stabilizer-subgroup}. 
    Now notice that $\Spin_{{\QF}_{|\orth{\Klein(\L)}}}(\KK) \simeq \HH_{\qoKL}(\KK) = \SO_{\qoKL}(\KK) ,~u\mapsto (x \mapsto ux)$, is an isomorphism.
    Since the map $\theta_{\oKLZ}$ is precisely the action of $\HH_\L$ on $\orth{\Klein(\L)}$ in a choice of basis, the above shows that the image of $\theta_{\oKLZ}\colon\HH_\L(\KK)\to\HH_{\qoKL}(\KK)$ is the set of squares.

    Since $D$ is square-free, one may show as in Lemma~\ref{lem:relationtoclassgroups} that the finite abelian group $\tdquot{\HH_{\qoKL}(\QQ)}{\HH_{\qoKL}(\Adeles_f)}{\HH_{\qoKL}(\ZZhat)}$ factors onto the set of squares in the class group $\mathfrak{cl}_K$ of the biquadratic field $K =\QQ(i,\sqrt{D})$. 
    Also, the kernel is bounded in size independently of $D$.
    The $2$-torsion (and hence also the $4$-torsion) of $\mathfrak{cl}_{K}$ is bounded by $\D^{o(1)}$ by \cite[\S1]{halter-koch_satz_1972} \textemdash~see also \cite[(3.5)]{Gerth}. 
    This proves the lemma.
	\end{proof}

\begin{theorem}\label{thm:Duke}
Let $F$ be a number field and let $\mathbf{S}$ be a form of $\SL_2$ over $F$.
Then there exists $\delta>0$ with the following property.

For every $n \geq 1$ suppose that $\mathbf{T}_n$ is a non-trivial $F$-torus, $\psi_n\colon \mathbf{T}_n \to \mathbf{S}$ is a non-trivial homomorphism, and $g_n \in \mathbf{S}(\mathbb{A})$.
Assume that 
\begin{align}\label{eq:voldisc}
\vol\big(\mathbf{S}(F)\psi_n(\mathbf{T}_n(\mathbb{A}))g_n\big) 
\geq \disc\big(\mathbf{S}(F)\psi_n(\mathbf{T}_n)(\mathbb{A})g_n\big)^{\frac{1}{2}-\delta}.
\end{align}
and that $\disc(\mathbf{S}(F)\psi_n(\mathbf{T}_n)(\mathbb{A})g_n)\to \infty$.
Then $\mathbf{S}(F)\psi_n(\mathbf{T}_n(\mathbb{A}))g_n$ is equidistributed in $\mathbf{S}(F)\backslash \mathbf{S}(\mathbb{A})$ as $n\to \infty$.
\end{theorem}

\begin{proof} 
    We sketch the principles of the proof which is certainly well known to the experts.
    
    By Weyl's equistribution criterion, it suffices to show that for any automorphic form $\varphi_1\in\pi_1\subset L^2([\mathbf{S}])$ which is not the constant function, one has
    \[ \int_{[\psi_n(\mathbf{T}_n)]}\varphi_1(tg_n)dt\to 0,\ n\to\infty. \]
    Here, $[\psi_n(\mathbf{T}_n)]$ denotes the quotient \[ [\psi_n(\mathbf{T}_n)]=\mathbf{S}(F)\psi_n(\mathbf{T}_n(\mathbb{A}))\subset \mathbf{S}(F)\backslash \mathbf{S}(\mathbb{A}). \]
    
    Let $\mathbf{PB}^\times$ be the group of projective units of the quaternion algebra associated with $\mathbf{S}$ and let $\mathbf{T}_n'\subset \mathbf{PB}^\times$ be the torus corresponding to $\psi_n(\mathbf{T}_n)$. 
    It follows from the work of Jacquet-Langlands and Labesse-Langlands \cite{JL,LL,Lab} that there exists an automorphic form 
    \[ \varphi\in \pi\subset L^2(\mathbf{PB}^\times(F)\backslash \mathbf{PB}^\times(\mathbb{A})) \] whose restriction (or pull-back) to $\mathbf{S}(\mathbb{A})$ is $\varphi_1$. 
    
    The above integral then equals
    \begin{equation}
    	\label{period1}
    	\int_{[\psi_n(\mathbf{T}_n)]}\varphi(t^2g_n)dt
    \end{equation}
    where (abusing notation)
   \begin{align*}
    [\psi_n(\mathbf{T}_n)]=\mathbf{PB}^\times(F)\psi_n(\mathbf{T}_n(\mathbb{A}))\subset [\TT_n']=\mathbf{PB}^\times(F)\mathbf{T}_n'(\mathbb{A})\subset \mathbf{PB}^\times(F)\backslash \mathbf{PB}^\times(\mathbb{A}).
   \end{align*}
    
    Let $D_n$ be the absolute value of the $F/\QQ$-norm of the discriminant of the quadratic order associated to the torus orbit $[\mathbf{T}_n']g_n$. Because of assumption \eqref{eq:voldisc}, the characteristic function of $[\psi_n(\mathbf{T}_n)]$ inside $[\mathbf{T}_n']$ can be expressed as a linear combination of a group of $O_{\delta'}(D_n^{\delta'})$ finite order characters  of $\mathbf{T}_n'(F)\backslash \mathbf{T}_n'(\mathbb{A})$ (whose coefficients are bounded by~$1$) for any $\delta'>\delta$; we are then reduced to proving that
    \begin{equation}\label{decay}
    	D_n^{\delta'}\max_{\chi} \Big|\int_{[\mathbf{T}_n']}\varphi(t^2g_n)\chi(t)dt \Big|\to 0
    \end{equation}
    where $\chi$ is varying over that group of characters.

    Likewise, the characteristic image of the morphism $t\mapsto t^2$ in $[\mathbf{T}_n']$ is a linear combination of characters of the genus (order at most $2$) of $[\mathbf{T}_n']$ which is of size $O_\varepsilon(D_n^\varepsilon)$ for any $\varepsilon>0$, so it suffices to prove that
    \begin{equation}\label{decay2}
    	D_n^{\delta'+\varepsilon}\max_{\chi'} \Big|\int_{[\mathbf{T}_n']}\varphi(tg_n)\chi'(t)dt \Big|\to 0
    \end{equation} 
    for $\chi'$ varying over the group generated by the characters appearing in \eqref{decay} and the genus 
    characters.
    
    The integral in \eqref{decay2} can be evaluated using Waldspurger's formula \cite{Walds}: for $\varphi$ a factorable vector in $\pi$ (which we may assume) one has
    \[ \Big|\int_{[\mathbf{T}_n']}\varphi(tg_n)\chi'(t)dt\Big|^2=D_n^{o(1)}\frac{L(1/2,\pi\otimes \theta_{\chi'})}{D_n^{1/2}}\prod_{v}I_v(\mathbf{T}_n',g_v,\varphi_v), \]
    where $L(s,\pi\otimes \theta_{\chi'})$ is the Rankin-Selberg $L$-function of $\pi$ and the theta series representation associated with $\chi'$ and  the factors $I_v(\mathbf{T}_n',(g_n)_v,\varphi_v)$ are local toric integrals:
    \[ I_v(\mathbf{T}_n',(g_n)_v,\varphi_v)=\int_{\mathbf{T}_n'(F_v)}\frac{\langle t.g_v.\varphi_v,g_v.\varphi_v\rangle}{\langle\varphi_v,\varphi_v\rangle}dt. \]
    
    Subconvex bounds for $L(1/2,\pi\otimes \theta_{\chi'})$ from  \cite{MV} and bounds for the local integrals $I_v(\mathbf{T}_n',(g_n)_v,\varphi_v)$ from \cite{CU} give that
    \[ \Big|\int_{[\mathbf{T}_n']}\varphi(tg_n)\chi'(t)dt\Big|\ll_{\varphi}D_n^{-\eta} \] for some absolute constant $\eta>0$. It follows that \eqref{decay} holds as soon as $\eta\geq \delta'+2\varepsilon$ which can be achieved for $\delta<\eta$ by taking $\varepsilon$ sufficiently small.

\end{proof}

    \begin{proof}[Proof of Proposition~\ref{prop:indiv-equi}]
        For $1\leq j \leq 4$ the pushforward of $\mu_{\L_n}$ to $\XX_{p,j}$ is equal to the pushforward of the normalized Haar measure on $\GG(\QQ)\TT_{\L_n}(\Adeles)y_{\L_n}$ to $\XX_{p,j}$ and so we may as well study the latter. Further, the diagram equipped with natural projections
        \[
            \begin{tikzcd}
                \XX_\Adeles \arrow[r] \arrow[d] & \XX_{\Adeles,j} \arrow[d] \\
                \XX_p \arrow[r]  &\XX_{p,j}
            \end{tikzcd}
        \]
        commutes and thus we may consider the projection $\XX_\Adeles \to \XX_{\Adeles,j}$ to the individual components first. For $j=1$, the projection of $\GG(\QQ)\TT_{\L_n}(\Adeles)y_{\L_n}$ to $\XX_{\Adeles,1}$ is equal to $\GG_1(\QQ)\HH_{\L_n}(\Adeles)k_{\L_n}$. By \Cref{lem:relationtoclassgroups}, this projection is given by the finitely many $k_{\L_n}^{-1}\HH_{\L_n}(\RR\times\ZZhat)k_{\L_n}$-orbits corresponding to the squares in the class group of $\QQ(i,\sqrt{D_n})$, where $D_n =\disc(\L_n)$. Since the set of squares has index $D_n^{o(1)}$, the pushforward of the normalized Haar measure on $\GG(\QQ)\TT_{\L_n}(\Adeles)y_{\L_n}$ to $\XX_{\Adeles,1}$ equidistributes as $n\to\infty$ by Theorem~\ref{thm:Duke}. 
        Thus, the projection of $\mu$ to $\XX_{p,1}$ is equal to the Haar measure. For $j=2$, the projection of the shifted adelic torus orbit $\GG(\QQ)\TT_{\L_n}(\Adeles)y_{\L_n}$ to $\XX_{\Adeles,2}$ is given by $\GG_2(\QQ)\theta_{\L_n}(\HH_{\L_n}(\Adeles))M_{\L_n}$. By \Cref{lem:large-packets}, this projection is given by the finitely many $M_{\L_n}^{-1}\theta_{\L_n}(\HH_{\L_n}(\RR\times\ZZhat))M_{\L_n}$-orbits containing the ones corresponding to the squares in the class group $\mathfrak{cl}_{\q_{\L_n}}$. The set of squares has index $D_n^{o(1)}$ and so we may argue as above. The analogous argument holds for $j=3$. For $j=4$ we use \Cref{lem:image-of-third-morphism} and argue as in the case $j=1$.
        By definition, the measures $\mu_{\L_n}$ are invariant under the action of the subgroup $\{(a_s,e,a_{2s},a_{2s})\colon s\in\RR\}$ and thus the limit $\mu$ is invariant under this action as well. 
    \end{proof}

\subsection{Proof of \texorpdfstring{\Cref{thm:dynamical-version}}{thm}}\label{sec:proof-dyn}

We first \enquote{decouple} the second factor $\XX_{p,2}$ from the other factors.
Let $\GG' = \GG_1\times\GG_3\times\GG_4$ and define the quotient $\XX_{p}'$ similarly.
Let $\mu'$ be the pushforward of $\mu$ to $\XX_p'$.

\begin{lemma}\label{lem:forget-second}
If $\mu'$ is the Haar probability measure on $\XX_p'$, then $\mu$ is the Haar probability measure on $\XX_p$.
\end{lemma}

    We use the following elementary and standard ergodic-theoretic lemma; for a proof see for example \cite[Lemma 7.2]{aka_equidistribution_2021}.

    \begin{lemma}\label{lem:joinings-for-ergodic-and-trivial}
        Let $\mathsf{Y}_1 = (Y_1,\mathcal{B}_1,T_1,\mu_1)$ and $\mathsf{Y}_2 = (Y_2,\mathcal{B}_2,T_2,\mu_2)$ be two measure-preserving dynamical systems. Suppose that $\mathsf{Y}_1$ is ergodic and that $\mathsf{Y}_2$ is trivial. Then the only joining of $\mathsf{Y}_1$ and $\mathsf{Y}_2$ is the trivial joining.
    \end{lemma}
    
    \begin{proof}[Proof of Lemma~\ref{lem:forget-second}]
    By assumption $\mu$ is a joining of the Haar measures for the action of $\{(a_s,a_{2s},a_{2s}):s\in \RR\}$ on $\XX_p'$ and the trivial action on $\XX_{p,2}$. The former action is ergodic by a theorem of Howe--Moore and, thus, the lemma follows from Lemma~\ref{lem:joinings-for-ergodic-and-trivial}.
    \end{proof}

    In the remainder of the proof of Theorem~\ref{thm:dynamical-version} we show that $\mu'$ is the Haar probability measure on $\XX_p'$.
    We begin by exhibiting an additional diagonalizable invariance at the prime $p$ (so that $\mu$ is invariant under a higher-rank diagonalizable action).

    \begin{lemma}\label{prop:split-torus}
    Let $q$ be a binary quadratic form over $\ZZ_p$ so that $-\disc(q)$ is a non-zero square modulo $p$. Then $\Spin_q$ is split over $\QQ_p$. 
    \end{lemma}

    \begin{proof}
    By an application of Hensel's lemma (see e.g.~\cite[Ch.~4]{cassels2008rational}), $q$ is isotropic over $\QQ_p$.
    \end{proof}
    
    Recall the definition of \highlight{class-$\mathcal{A}'$} actions from~\cite[Def.~1.3]{einsiedler_joinings_2019}.

    \begin{lemma}\label{lem:class-A-action}
        If $D_n$ is a non-zero square mod $p$ for all $n$, there is a homomorphism $\iota_p:\ZZ \to \GG'(\QQp)$ of class-$\mathcal{A}'$ whose image acts ergodically with respect to the Haar measure on $\XX_p'$ and preserves $\mu'$.

        If $-D_n$ is a non-zero square mod $p$ for all $n$, there is a homomorphism $\iota_p:\ZZ \to \GG'(\QQp)$ so that the projection of the image to $(\GG_1 \times \GG_4)(\QQ_p)$ is of class-$\mathcal{A}'$ and acts ergodically with respect to the Haar measure on $\XX_{p,1}\times\XX_{p,4}$, the projection of the image to $\GG_3(\QQ_p)$ is trivial, and $\iota_p(\ZZ)$ preserves $\mu'$.
	\end{lemma}
	
	\begin{proof}
    We show invariance under a diagonally embedded stabilizer group of a $p$-adic plane $E_p$, where in the second case of the lemma this group acts trivially on the third factor.
    
	Recall the choices $\BT{\L_n},\BTo{\oL_n}\in\SL_4(\ZZ)$ and $\BToK{\L_n}\in\SL_3(\ZZ{[i]})$ from \S\ref{sec:shapes}. 
    Since $\SL_4(\ZZp)$ and $\SL_3(\ZZp{[i]})$ are compact, we may assume that $\BT{\L_n}\!\to\!\operatorname{M}\in\SL_4(\ZZp)$, $\BTo{\oL_n}\!\to\orth{\operatorname{M}}\in\SL_4(\ZZp)$ and $\BToK{\L_n}\!\to\operatorname{N}\in\SL_3(\ZZp{[i]})$ after passing to a subsequence.
    Let $E_p$ the oriented $\QQp$-plane spanned by the first two rows of $\operatorname{M}$.
    In the first case of the lemma, let $q$ be the quadratic form on the $\ZZ_p$-span of the first two rows of $\orth{\operatorname{M}}$ (these span the complement of $E_p$) and in the second case, let $q$ be the quadratic form on the $\ZZ_p$-span of the first two rows of $\operatorname{M}$. 
    In either case, $-\disc(q)$ is a non-zero square modulo $p$ by assumption.
    By \Cref{prop:split-torus}, we have that $\HH^{pt}_{E_p}\simeq \Spin_q$, respectively $\HH^{pt}_{E_p^\perp}\simeq \Spin_q$ is split over $\QQ_p$.

    Set
    \begin{align*}
    \theta_{E_p^\perp}(h)= \pi_{\lowerright}(\operatorname{M}^\perp[\Adj_h](\operatorname{M}^\perp)^{-1}),\quad
    \theta_{\Lambda_{E_p}}(h) = \pi_{\slowerright}(\operatorname{N}[[\Adj_h]]\operatorname{N}^{-1}).
    \end{align*}
    By construction, $\mu'$ is invariant under $\Theta_{E_p}'(h) \defeq (h,\theta_{E_p^\perp}(h), \theta_{\Lambda_{E_p}}(h))$ for every $h \in \HH_{E_p}(\QQp)$.
    In the first, respectively second, case of the lemma, we take $\Theta_{E_p}'$ restricted to an unbounded cyclic subgroup of $\HH^{pt}_{E_p}(\QQ_p)$, respectively $\HH^{pt}_{E_p^\perp}(\QQ_p)$.
    With these choices, the lemma follows.
	\end{proof}
        
	\begin{proof}[Proof of \Cref{thm:dynamical-version}]
        If $-D_n$ is a non-zero square modulo $p$ for all $n$, we may apply Lemma~\ref{lem:joinings-for-ergodic-and-trivial} one more time and reduce the theorem to proving that the pushforward of $\mu$ to the product of the first two factors is the corresponding Haar measure.
        The following proof considers the case where $D_n$ is a non-zero square modulo $p$ for all $n$; the complementary case is completely analogous.
    
		We want to apply the joining classification \cite[Thm.~1.4]{einsiedler_joinings_2019} to $\mu'$. 
        As $\GG'$ is semisimple and simply connected, $\XX'_p$ is saturated by unipotents in the sense of \cite[Def.~1.1]{einsiedler_joinings_2019}. 
        That is, the group generated by all unipotent elements of $\GG'(\RR\times\QQ_p)$ acts ergodically on $\XX'_p$.
		
		We decompose $\mu'$ into ergodic components for the $\ZZ^2$-action given by $\iota:(m,n)\mapsto ((a_m,a_{2m},a_{2m}),\iota_p(n))$ (see~\Cref{lem:class-A-action}) and show that almost all of these ergodic components are the trivial joining for this action. Let $\nu$ be an ergodic component of $\mu'$. 
        By \cite[Cor.~1.5]{einsiedler_joinings_2019}, if the projection $\nu_{j,k}$ of $\nu$ to the $j$-th and $k$-th coordinate is equal to the Haar measure $m_{X_j}\times m_{X_k}$ for all $j<k$ with $j,k\in\{1,3,4\}$, then $\nu = m_{\XX'_p}$.
		Define $\GG_{jk} \defeq \GG_j\times\GG_k$ for $j,k\in\{1,3,4\}$ and set $G_{jk} \defeq \GG_{jk}(\RR\times\QQ_p)$ and $\Gamma \defeq \GG_{jk}(\ZZ[\tfrac{1}{p}])$.
        Applying \cite[Thm.~1.4]{einsiedler_joinings_2019} to $\nu_{j,k}$ yields that $\nu_{j,k}$ is algebraic and defined over $\QQ$, that is, there exists a $\QQ$-algebraic group $\mathbf{M}\leq \GG_{jk}$ and a finite-index subgroup $N\leq\mathbf{M}(\RR\times\QQ_p)$ such that $\nu_{j,k}$ is the normalized Haar measure on a single orbit $\Gamma Ng$ for some $g\in G_{jk}$. 

		In case that $\mathbf{M} = \GG_{jk}$, $G_{jk}$ does not contain any proper finite-index subgroup since $\Res{\QQi}(\SL_2)$ and $\SL_2$ are simply connected.
        Thus, $N = G_{jk}$ and $\nu_{j,k}$ is the trivial joining.
        
		Assume now that $\mathbf{M}$ is a proper subgroup of $\GG_{jk}$. Since $\nu_{j,k}$ is a joining and $\GG_j$ and $\GG_k$ are simply connected, $\mathbf{M}$ projects surjectively to $\GG_j$ and $\GG_k$, that is, $\mathbf{M}$ is the graph of some $\QQ$-isomorphism $\psi:\GG_j\to \GG_k$ by a Goursat-type lemma. Thus, the cases $(j,k) \in\{(1,3),(3,4)\}$ immediately lead to a contradiction.
		Assume $(j,k) = (1,4)$. Then $\GG_j = \GG_k = \Res{\QQi}(\SL_2)$ and $N = \psi(\GG_{j}(\RR\times \QQ_p))$.
        We know that $N_\infty$ contains $\{(a_s,a_{2s})\colon s\in\RR\}$. 
        However, $(a_s,a_{2s})$ cannot lie on the graph of the isomorphism $\psi$, since the eigenvalues of $a_s$ and $a_{2s}$ do not coincide. Hence, such an isomorphism $\psi$ cannot exist. This concludes the proof of \Cref{thm:dynamical-version}.
	\end{proof}
    
    \bibliographystyle{amsplain}
    \bibliography{database}
\end{document}